\documentclass[12pt]{article}
\usepackage{amsfonts,amsmath,amssymb,textcomp,enumerate,bbm}
\usepackage{amsthm}
\usepackage{comment}
\usepackage[colorlinks = true,
            linkcolor = blue,
            urlcolor  = blue,
            citecolor = blue,
            anchorcolor = blue]{hyperref}

\newcommand{\R}{\mathbb{R}}
\DeclareMathOperator{\Ric} {Ric}

\DeclareMathOperator{\Id} {Id}
\newtheorem{theorem}{Theorem}

\newtheorem{lemma}{Lemma}
\newtheorem{corollary}{Corollary}
\newtheorem{remark}{Remark}
\newtheorem{definition}{Definition}
\newtheorem{hypothesis}{Hypothesis}

\newcommand{\F}{\mathcal{F}}
\newcommand{\E}{\mathbb{E}}
\newcommand{\mathd}{\mathrm{d}}
\usepackage[english]{babel}
\usepackage{color}
\allowdisplaybreaks

\newcommand{\assign}{:=}

\newcommand{\tmop}[1]{\ensuremath{\operatorname{#1}}}

\title{Continuous sparse domination and dimensionless weighted estimates for the Bakry Riesz vector} 
\author{K. Domelevo, S. Petermichl\thanks{partially supported by the Alexander von Humboldt Stiftung and ERC grant CHRiSHarMa no. DLV-682402}, K.A. \v Skreb\thanks{partially supported by Croatian Science Foundation project UIP-2017-05-4129 (MUNHANAP) and ERC grant CHRiSHarMa no. DLV-682402}} 
\date{}

\begin{document} 

\maketitle
\begin{abstract}
We present a fundamentally new proof of the dimensionless $\textup{L}^p$ boundedness of the Bakry-Riesz vector on manifolds with bounded geometry. Our proof has the significant advantage that it allows for a much stronger conclusion than previous arguments, namely that of some new dimensionless weighted estimates with optimal exponent. Part of the importance of this task lies in the novelty of the techniques: we develop the self-similarity argument known as `sparse domination' in the setting of uniformly integrable c\`adl\`ag  Hilbert space valued martingales and extend the domination to a process with infinite memory. We provide a range of optimal weighted estimates and weak type estimates for these stochastic processes. Previous geometric Riesz transform estimates relied on Bellman functions and did not provide this range of weighted estimates. The development of sparse domination in this probabilistic setting and its use for high dimensional problems is new. 
\end{abstract}

\section{Introduction}

\subsection{A short summary}
In this paper, we are primarily interested in dimensionless weighted and unweighted norm estimates of the Riesz vector on manifolds as well as related stochastic processes that are used to model these operators. These are  loosely tied to the concept of differential subordination. 
The ties of pairs of martingales in differential subordination to harmonic analysis have a long history and have proven influential, 
especially with ambitious goals such as very precise, sharp norm estimates. 
 Certain classical operators in harmonic analysis, such as the Euclidean Riesz transforms can be written as a conditional expectation of
certain martingale transforms such as done in Gundy-Varopoulos {\cite{GV1979}}. Other, deep connections 
of a probabilistic flavour have been established long ago in the work by Bourgain and Burkholder \cite{Bou} \cite{Bur}. More surfaced in the last ten to twenty years by one of the authors {\cite{Pet2000b}}. 
Especially in the presence of negative curvature, the use of randomness is natural {\cite{Emery}}. Other than previous arguments, only a small part of our proof is based on special auxiliary (Bellman) functions, while the core of the argument is a weak type estimate and a very particular continuous-time sparse domination of the stochastic process by X.D. Li \cite{Li2008}, whose projection is the Bakry Riesz vector {\cite{Bakry1985}}, {\cite{Bakry1987}} of a given function.

Therefore, the aim of this paper is two fold, there are advances in harmonic analysis in a geometric setting as well as in probability theory:

\begin{itemize}

\item The results in analysis concern dimensionless estimates for the so called Bakry Riesz vector on Riemannian manifolds with their Bakry-Emery curvature bounded below. In the non-negative curvature case and for all $L^p$, $1<p< \infty$ it is shown that there holds even in the weighted setting, a dimension-free bound. The weighted bound was previously only known for $p=2$, with a completely different proof that has no hopes to allow for an extension to other $p$. The estimate we obtain is sharp, in the sense that the exponent on the weight characteristic cannot be improved. (See the Euclidean case \cite{KT}). In the negative curvature case we recover known unweighted results by \cite{BO}, \cite{CD} with a markedly different proof, avoiding Bellman functions almost entirely (We use and prove a weak type estimate that requires a Bellman function).  Throughout, we apply a novel approach via continuous sparse domination instead of Bellman functions. This is the first use of sparse domination using continuous-time probabilistic flow techniques with the goal to obtain structural information on objects central to high dimensional harmonic analysis.

\item Our results on the probabilistic side center around a development of sparse domination using a continuous time parameter. This results in sharp weighted estimates for relevant stochastic processes, even in the presence of jumps. The simplest estimate we present is a weighted maximal estimate for processes with jumps, a question raised in the 70s \cite{BL}. It is not difficult, with today's techniques, to answer this question positively. We move on to an optimal weighted bound for the maximal function of a martingale $Y$ with respect to $X$, a pair of differentially subordinate martingales. In the case of dyadic martingales with underlying Lebesgue measure, the result is due to Wittwer \cite{W} and was considered a difficult accomplishment at that time. More issues arise when the regularity of the dyadic system is lost. In this general discrete time case, such results are fairly recent by Thiele-Treil-Volberg \cite{TTV} (without maximal inequalities) and Lacey \cite{Lacey2015} (gives maximal inequalities using the powerful novel technique of sparse domination). In the case of continuous time but under minimal regularity assumptions on the paths, the best previously known result by two of us uses a rather involved Bellman function, has a technically demanding proof \cite{DP1} and cannot treat maximal inequalities. In this current paper, we provide a sparse domination and weighted estimate for the maximal function of $Y$ in greatest generality. (This martingale model via the pair $X$ and $Y$ suffices to deduce Riesz vector estimates in the Euclidean case but is insufficient for Riesz vectors on general manifolds). The development of a sparse domination in this continuous time case is new, but the trained probabilist might be able to develop it from the ideas of Lacey \cite{Lacey2015} with the appropriate changes, recalling that jump processes can be a mine field. 
In the current situation, though, one needs a model $Z$ for the Riesz vector that is merely a semi-martingale with an infinite memory component defined via a stochastic ODE. The stochastic process associated to $f$ is a semi-martingale of the form $X+A$ involving a martingale $X$ and a finite variation process $A$, and the process $Z$ is directly related to a martingale transform $Y$ of $X$, where $Y$ is differentially subordinate to $X$ via the square bracket process.
It appears that $Z$ is a process with infinite memory, therefore incompatible with the standard sparse domination. We introduce a new, more subtle domination, preserving both sparsity and self similarity.
We still get  the needed weak type estimate and sparse domination for $Z$ and therefore maximal weighted bounds for all $p$, even in the presence of jumps. The weighted bounds as well as the sparse domination for $Z$ are the main results in the probability section of this paper.

\end{itemize}

\subsection{Detailed introduction}

In the Euclidian space, the $i$--th Riesz transfrom in $\R^n$ is defined as
\[
        R_i = \frac{\partial}{\partial x_i} (-\Delta)^{-1/2},
\]
where $\Delta = \sum_{i=1}^n \partial^2 / \partial_{x_i}^2$ is the usual Laplacian in $\R^n$. The vector Riesz transform $\vec{R}$ is defined as the vector valued operator $\vec{R}=(R_1,R_2,\ldots,R_n)$. In the one-dimensional setting, the Riesz transform is nothing but the Hilbert transform. The $\textup{L}^p$ estimate of the Hilbert transform on the real line dates back to the work of Riesz \cite{Riesz1927} and Pichorides \cite{Pic1972}. Regarding the $\textup{L}^p$ estimate of the Riesz vector in $\R^n$, see \cite{Stein1955, Meyer1984, Pisier1988, BanWan1995, IwaMar1996, DraVol2006}. The knowledge of the exact value (or at least a good estimate) of the $p$-norm of the Riesz transforms on $\R^n$ is a recent matter whose importance appears in the theory of quasiconformal mappings and related PDEs.

A corner stone in this line of results is the stochastic representation of Riesz transfroms in $\R^n$ by Gundy-Varopoulos \cite{GV1979}. To this end, these authors define the so-called background noise which are, morally, Brownian trajectories in the upper half space started at infinity and stopped when hitting the boundary. To a given function $f$ defined on $\R^n$ and its Poisson extension in the upper half space, these authors associate a natural martingale $M^f$. They prove that the Riesz transforms can be written as a suitable conditional expectation of martingale transforms of $M^f$. This representation was extended to the Riemannian manifold framework by X.-D. Li \cite{Li2008,LiErratum,LiArXiv}. Despite the loss of the martingale structure in the process that models the Riesz vector, it has thus enabled the first dimensionless estimates with the growth proportional to $(p-1)^{-1}$ when $p<2$ and $p-1$ when $p\geqslant 2$ in this setting \cite{BO}. A somewhat earlier complete proof via a different, deterministic semi-group method can be found in \cite{CD}. Both arguments are fairly recent Bellman function arguments, one in the probabilistic `strong form' and one formulated in the deterministic `weak' or `dualized form'.

For early considerations of $\textup{L}^p$ boundedness of Riesz transforms on manifolds, we refer to \cite{SteinTopics}. We mention also the works \cite{Meyer1984,Bakry1985,Gundy1986,Bakry1987,Pisier1988,Arcozzi,BanBau2013} among which the papers of Bakry provide estimates of Riesz transforms for complete Riemannian manifolds under the general condition that the Bakry-Emery curvature is bounded below (see \cite{Emery}). 

In this paper, we recover the results mentioned above in \cite{CD} and \cite{BO}. Our proofs, however, use a novel technique and work in the weighted setting. In this light, we generalize the Euclidean weighted result obtained in \cite{DPW} by working on manifolds, and the one obtained in \cite{D} by extending the estimates to $\textup{L}^p(w)$, for all $p\in(1,\infty)$.
Weighted estimates of the type we obtain here tend to be much stronger than unweighted estimates. 

Usually, it is possible to extrapolate from a given exponent to other exponents (as opposed to interpolate in the unweighted case) but the classical result on extrapolation of Rubio de Francia cannot be applied here directly to the Riesz vector, because the manifold is endowed with a weighted measure $$\mathd\mu_\varphi(x)=e^{-\varphi(x)}\mathd\nu(x)$$ that is not necessarily doubling. It is also to be noted that with the semi-group approach such as in \cite{PW} and \cite{D}, extrapolation is also not useful, not even in the Euclidean case! (This situation changes in the model parabolic case, where the heat kernel presents much stronger decay.) Further it is known to the authors that the Bellman function developed in \cite{DP1} and used in \cite{D} does not have the convexity properties needed to handle negative curvature. An $L^p$ expression for this function is unfortunately unknown. We also mention the recent beautiful construction in \cite{BBO} that also handles the case of $p=2$ differently from \cite{DP1} and gives maximal estimates. It falls short of treating jump processes, but gives another valuable less complicated construction than that in \cite{DP1}.

Our proof is very different from those in previously appeared results in that it does not rely on a Bellman function for the problem.
We will thus use stochastic tools relying on the stochastic representation of Riesz transforms on manifolds by X.-D. Li \cite{Li2008,LiErratum,LiArXiv}. To this end, we use the Eells-Elworthy-Malliavin construction of Brownian motion on manifolds \cite{Mal1997, Elworthy}. It is defined as the projection on the manifold of the solution of an SDE (called horizontal Brownian motion), defined on the orthonormal frame bundle $\mathcal{O}(M)$. This construction allows to describe diffusion processes on manifolds and gives rise to the notion of parallel transport along the paths of the diffusion \cite{Elworthy, Hsu}.

For readers not familiar with sparse domination, we refer the reader to the elegant original argument in \cite{Lacey2015} for the first probabilistic object, a discrete time martingale transform. For ease of transition, we develop the argument rigorously in the continuous time martingale setting. Readers familiar with \cite{Lacey2015} but new to stochastic processes can use the proximity of these arguments to get acquainted. Both the required weak type estimate and in particular the sparse domination become more complicated when one loses the martingale structure. The process $Z$ by Li is defined via an auxiliary martingale $Y$ that is differentially subordinate to $X$, but $Z$ is not a martingale. (Here, $X$ models $f$ and an expectation of $Z$ models $Rf$.)  The process $Z$ is defined via $Z_0=0$ and evolves according to the stochastic ODE $$\mathd Z_t=(V_t-a\Id)Z_t \mathd t + \mathd Y_t.$$
(The sign of the operator $(V_t-a \Id)$ is negative and it will be explained later how it relates to the curvature of the space.) The ODE above defines a process with memory, while classical sparse type arguments are self similar and resist to dominating objects with memory. A new type of domination procedure comes into play to overcome this difficulty, relying on the specific form of the ODE. Another needed tool is a weak type estimate of the maximal operator of $Z$. This is the only part of our proof that uses a (simple) Bellman function. Similar processes have been considered by Ba\~nuelos-Osekowski \cite{BO}, in our estimate we need to consider also submartingales. The explicit form of the function is essential and not just its convexity and size properties. The first derivative of said Bellman function is used to control a drift term that arises because the process we consider is not a martingale. 

The sparse domination technique has witnessed considerable efforts in the last several years and has been used to prove numerous new results in harmonic analysis, using sparse operators defined on cubes. These cannot give dimensionless estimates and are often not adapted to the non-doubling case. Our use of stochastic processes and stopping times with a continuous parameter to tackle problems in geometric harmonic analysis is entirely new. One can deduce, in general, from such a domination a dimensionless bound. The sparse operators are particularly well suited for working with weights, which is why this so obtained dimensionless estimate also holds in the weighted setting.

\section{Bakry Riesz transforms on manifolds.}

Let $(M, g)$ be a complete Riemannian manifold with metric $g$ and of dimension
$n$. Let $\Delta$ be the non-positive Laplace-Beltrami operator, $\nabla$
the gradient operator, and $\nu$ the volume measure on $(M, g)$ such that
$$\mathd  \nu (x) = \sqrt{\det g (x)} \mathd  x.$$ Let moreover $\mu_{\varphi}$ be a
weighted volume measure on $M$ with $$\mathd  \mu_{\varphi} (x) = e^{- \varphi (x)} \mathd 
\nu (x),$$ where $\varphi (x) \in \mathcal{C}^2 (M)$. The weighted Laplacian
$\Delta_{\varphi}$ with respect to $\mu_{\varphi}$ on $M$ is defined for any function
$f$ by
\[ \Delta_{\varphi} f \assign \Delta f- \nabla \varphi \cdot \nabla f. \]

We assume that $\mu_{\varphi} (M) < \infty$ and by normalizing, we may assume without
loss of generality that $\mu_{\varphi}$ is actually a probability measure. The
Bakry-Emery curvature tensor associated with $\Delta_{\varphi}$ is defined by
\[ \tmop{Ric}_{\varphi} = \tmop{Ric} + \nabla^2 \varphi, \]
where $\tmop{Ric}$ denotes the Ricci curvature tensor on $M$ and $\nabla^2
\varphi$ the Hessian of $\varphi$ with respect to the Levi-Civita connection
on $(M, g)$.

All over this paper, we assume that the Bakry-Emery curvature is bounded from below, i.e. there exists a non-negative constant $a$ such that
\[
	\tmop{Ric}_{\varphi}  \geqslant -a.
\]
This condition ensures that the manifold is stochastically complete and hence Brownian motion does not explode in finite time. We denote by $\vec{R}_\varphi^a$ the Bakry Riesz transform defined as
\[
	\vec{R}_\varphi^a = \nabla \circ (a\Id -\Delta _{\varphi} )^{-1/2},
\]
where $\Id$ denotes the identity. When $a=0$, we will denote the Riesz transforms just by $\vec{R}_\varphi$. It was proved already by Bakry \cite{Bakry1987} that for any $p>1$ there exists a universal constant $C_p$ 
such that for any function $f\in C_0^\infty(M)$, there holds
\[
	\|\vec{R}_\varphi^a  f \|_{\textup{L}^p(\mu)} \leqslant C_p \| f\|_{\textup{L}^p(\mu)}.
\]
However, the estimate did not have the asymptotic in $p$ one would have expected. The more recent works mentioned earlier provided the correct growth in $p$ near the end points.

\subsection{Probabilistic representation of the Bakry Riesz vector on manifolds.}

For this application, the familiar setting of continuous path processes is sufficient. Readers not familiar are encouraged to draw the relevant definitions from the next chapter, where they are rigorously given in greater generality. For the time being, we recall that the martingale $Y$ is said to be \emph{differentially subordinated} to the martingale $X$ if the process $([ X,X]_t -  [Y,Y] _t)_{t\geqslant 0}$ is non-negative and non-decreasing in $t$, where the bracket $[ \cdot,\cdot ]$ denotes the usual quadratic covariation process for real or vector valued stochastic processes. 

Let $B^M_t$ be the diffusion process on $M$ with generator $\Delta_{\varphi}$ and initial distribution $\mu$ obeying the stochastic differential equation
\[
	\mathd B^M_t = U_t \mathd W_t - \nabla\varphi(B^M_t) \mathd t,
\]
where $W_t$ is the Brownian motion on $\R^n$ and $$U_t:T_{B_0^M}M\to T_{B_t^M}M$$ denotes the stochastic parallel transport on $M$ along the trajectory $\{  B^M_s  , 0\leqslant s \leqslant t\}$. Let further $B_t$ be the one-dimensional Brownian motion started at $y>0$ with the normalisation $E [(B_t)^2] = 2t$ such that its generator is $\partial^2/\partial y^2$.
Following \cite{Meyer1984,GV1979, Gundy1986}, there exists a diffusion process $$\tilde{B}_t=(B^M_t,B_t)$$ on $(M,\R^+)$ -- the so-called \emph{background radiation process} -- associated to the generator $\Delta_\varphi + \partial^2/\partial y^2$ and with the initial distribution $\mu \otimes \delta_y$.

Using a martingale approach, one can represent the Riesz vector $\vec{R}_\varphi^a$ via a probabilistic representation. In the literature, it first appeared in \cite{GV1979}, where the Riesz transform was defined on $\mathbb{R}^n$. In \cite{Arcozzi} Arcozzi extended this formula to compact Lie groups and spheres. In \cite{Li2008}, \cite{LiErratum} X.-D.~Li presented a new formula adapted to complete Riemannian manifolds. The representation formula of the Riesz vector in this setting for a complete manifold with $\Ric _{\varphi}\geqslant -a$ is as follows

\begin{equation} \label{rep}
-\frac{1}{2}(\vec{R} _{\varphi}^a f)(x)=\lim_{y\rightarrow \infty} \mathbb{E}_y\left[e^{-a \tau}M_{\tau}\int_0^{\tau} e^{a s} M_s^{-1} \nabla Q^a f(B^M_s,B_s)\mathd B_s | B^M_{\tau}=x \right],
\end{equation}
where $y$ is the starting point of the random walk and
\begin{itemize}
\item $Q^a f (x,y)=e^{-y\sqrt{a\Id-\Delta_{\varphi}}}f(x)$ is the Poisson extension;
\item $\tau = \inf \{t>0:B_t=0\}$ is the stopping time upon hitting the boundary of the upper half space;
\item $M_t$ is the solution to the matrix-valued stochastic differential equation
\[\mathd M_t=V_tM_t \mathd t, \ \ \ M_0=\Id,\]
for some adapted and continuous process $(V_t)_{t \geqslant 0}$ taking values in the set of symmetric and non-positive $n \times n$ matrices and arises from the geometry of the manifold.
\end{itemize} 

Equivalently, one can rewrite this formula as
\begin{equation}
-\frac{1}{2}(\vec{R}_\varphi^a f)(x)=\lim_{y\rightarrow \infty} \mathbb{E}_y\left[ Z_\tau|B^M_{\tau}=x \right],
\label{rep2}
\end{equation}
where $Z_t$ is a semi-martingale defined thanks to the auxiliary martingale $Y_t$ as follows
\begin{gather*}
	Y_t =  \int_0^t  \nabla Q^af(B^M_s,B_s) \mathd B_s, \\
	Z_t =  e^{-at}M_t \int_0^t e^{as}M_s^{-1} \mathd Y_s,
\end{gather*}
where $(Y_t)_{t \geqslant 0}$ is by construction differentially subordinate to $(X_t)_{t \geqslant 0}$ defined as 
$$X_t = Q^af(B_0^M,y)+\int_0^t (\nabla,\partial_y) Q^af(B^M_s,B_s)(U_s \mathd W_s, \mathd B_s). $$
We will also define the auxiliary process
\begin{equation*}
\widetilde{X}_t := Q^af(B^M_t,B_t)= X_t + a \int_0^t Q^af(B^M_s,B_s) \mathd s,
\end{equation*}
where the second equality follows by It\^o formula (as can be seen in \cite{Li2008}). By denoting
\begin{equation}\label{eq: A}
A_t=a \int_0^t Q^af(B^M_s,B_s) \mathd s,
\end{equation}
we can write
\[\widetilde{X}_t=X_t+A_t \]
and observe that $(\widetilde{X}_t)_{t \geqslant 0}$ is in fact a submartingale. Indeed, since $(A_t)_{t \geqslant 0}$ is an increasing FV process and $(X_t)_{t \geqslant 0}$ is a martingale, we have
\begin{align*}
\E(\widetilde{X}_t|\F_s)&=\E(X_t+A_t|\F_s) \geqslant \E(X_t+A_s|\F_s) \\
&=\E(X_t|F_s)+\E(A_s|\F_s)=X_s+A_s=\widetilde{X}_s.
\end{align*}

For the theorems we wish to prove, it will be sufficient to consider $f\geqslant 0$ and thus $X$, $\widetilde{X}\geqslant 0$.

\begin{remark}
For the non-negative curvature case, i.e. $a=0$, the representation formula \eqref{rep} reduces to 
\begin{equation*}
-\frac{1}{2}(\vec{R} _{\varphi} f)(x)=\lim_{y\rightarrow \infty} \mathbb{E}_y\left[M_{\tau}\int_0^{\tau} M_s^{-1} \nabla Q f(B^M_s,B_s)\mathd B_s | B^M_{\tau}=x \right],
\end{equation*}
where $Q f(x,y)=e^{-y\sqrt{-\Delta_{\varphi}}}f(x)$. In this case obviously $A_t=0$ and $X_t=\widetilde{X}_t$ for every $t\geqslant 0$. In particular, that means that the process $(\widetilde{X}_t)_{t\geqslant 0}$ is a martingale.
\end{remark}

\section{Probabilistic objects in the presence of jumps and Sparse operators}

Let $(\Omega, \mathcal{F}, \mathfrak{F}, \mathbbm{P})$ be a complete filtered
probability space with $\mathfrak{F}= (\mathcal{F}_t)_{t \geqslant 0}$ a
filtration that is right continuous, where $\mathcal{F}_0$ contains all
$\mathcal{F}$ null sets. Let $X$ and $Y$ be uniformly integrable adapted
c{\`a}dl{\`a}g martingales with values in a separable Hilbert space that are
in a relation of differential subordination according to Ba\~nuelos-Wang \cite{BanWan1995} and Wang \cite{Wang} in the continuous time setting, and earlier Burkholder \cite{Bur1984a,Bur1988b,Bur1991a} in the discrete time setting.

\begin{definition}
  $Y$ is called differentially subordinate to $X$ if $[X, X]_t - [Y, Y]_t$ is
  non-negative and non-decreasing in $t$. In this case we also call the ordered pair $(X,Y)$ differentially subordinate.
\end{definition}

For the definition of the square bracket process and its properties, see for example
Dellacherie-Meyer {\cite{DM}} or Protter {\cite{Pro}}. Notice
that in particular $[X, X]_0 = | X_0 |^2$ so that differential subordination
of $Y$ with respect to $X$ implies $| Y_0 |^2 \leqslant | X_0 |^2$. 
Recall that for any stopping time $\tau$ the stopping sigma algebra is 
$$\mathcal{F}_{\tau} = \{ \Lambda
\in \mathcal{F}: \Lambda \cap \{ \tau \leqslant t \} \in \mathcal{F}_t \}.$$

\

To adapt the concept of a sparse set (usually of dyadic intervals), we make the following definition:

\begin{definition}
   An increasing sequence   $\{ T^j\}_{j\geqslant 0}$ of stopping times  with nested sets $E_j =
  \{ T^j < \infty \}$ is called sparse if 
$$\forall A^j \subset E_j, A^j \in \mathcal{F}_{T^j} {\text{ there holds }}
    \mathbbm{P} (A^j \cap E_{j + 1}) \leqslant \frac{1}{2} \mathbbm{P}
    (A^j).$$
\end{definition}

Notice that we do not have the simple plateaus as in the discrete time case and as it turns out, this definition using the sparseness localised to each subset measurable in the coarsest filtration is the correct replacement.

\begin{definition}
If $\{ T^j\}_{j\geqslant 0}$ is a sparse sequence and $E_j = \{ T^j < \infty \}$ its associated sequence of nested sets, we call
$$\mathcal{S}: X\mapsto \mathcal{S} ( X ) = \sum_{j = 0}^{\infty} | X |_{T^j} 
    \chi_{E_j} $$ a sparse operator.
\end{definition}

In cases, where some of the processes involved lose their martingale structure, this definition has to be slightly modified. In this context, for a non-negative process $\widetilde{X}=(\widetilde{X}_t)_{t\geqslant 0}$ we call the operator 
\begin{equation}\label{sparse1}
S(\widetilde{X}) = \sum_{j=0}^{\infty}\mathbb{E}(\widetilde{X}|\mathcal{F}_{T^j}) \chi_{E_j}
\end{equation}
a \emph{sparse operator} associated to $\{T^j\}_{j\geqslant 0}$. In the above definition we denote by $\widetilde{X}=\widetilde{X}_\infty$ also the closure of $\widetilde{X}$.

Observe that if the process $\widetilde{X}=X$ is also a non-negative martingale, the definitions coincide.

\begin{definition}
The maximal function associated with $X$ is  $$X^{\ast} = \sup_{t \geqslant 0} | X_t |.$$ 
\end{definition}

We shall throughout this text denote by
the same letters also the closures of the martingales that arise.

\

Last, we will often refer to the weight's probabilistic characteristic:

\begin{definition}
We call a positive uniformly integrable martingale $w$ a weight. The
quantity $Q_p (w)$ below is the $A_p$ characteristic of the weight $w$. If $Q_p (w)$ is finite, then we say $w \in A_p$.
\[ Q_p (w) = \sup_{\tau} \left\| \mathbbm{E} \left[ \left( \frac{w_{\tau}}{w}
   \right)^{\frac1{p-1}} \middle| \, \mathcal{F}_{\tau} \right]^{p - 1}
   \right\|_{\infty} = \sup_{\tau} \left\| w_{\tau} u^{p - 1}_{\tau} \right\|_{\infty} ,\]
where the supremum runs over adapted stopping times  $\tau$ and where 
we write $u^p w = u$.
\end{definition}

If necessary, we may put a superscript to emphasise the filtration: $Q^{\mathcal{F}}$.

\section{Main results}
 We use a sparse operator with continuous stopping times, dominating the martingale $Y_t$ as well as Li's process $Z_t$ whose projection is the Riesz vector. Instead of cutting the domain into cubes, we are therefore able to use the Poisson extension itself, thus resulting in clean dimensionless estimates. 

The estimate we aim to show will be a consequence of a sparse domination of the stochastic processes $Y_t$ and $Z_t$ (see \cite{NL,Lacey2015}). In the case of $Y_t$ the proof is relatively close to the argument of Lacey \cite{Lacey2015} but in the case of $Z_t$, the object is not a martingale, so the sparse domination is different and one key of the proof relies on the weak-$L^1$ estimate for the maximal function of the studied stochastic operator. A delicate new domination method also comes into play in a second step, utilising the defining stochastic differential equation.
We do not aim at the fullest generality here, keeping our goal in mind, an estimate for the Riesz vector. Here are the results concerning the Riesz vector:

We get a dimensionless unweighted estimate in $\textup{L}^p$ spaces for the Riesz vector on manifolds with curvature bounded from below (Theorem \ref{unweightedRiesz}). The first proofs of this result are recent \cite{CD}, \cite{LiErratum}, \cite{BO} and all based on a form of a Bellman function. All these cited Bellman proofs give a better numeric estimate than our proof, but as mentioned earlier, our proof extends (for free) to the weighted case, as long as the curvature is non-negative, which the previous proofs do not. Our $L^p$ estimates are linear in $p$, which means proportional to $(p-1)^{-1}$ when $p<2$ and to $p-1$ when $p\geqslant2$.

\begin{theorem}[$\textup{L}^p$ estimate for the Riesz vector]\label{unweightedRiesz}
Suppose that $M$ is a complete Riemannian manifold without boundary and $\Ric_{\varphi}\geqslant-a$ for some $a\geqslant 0$. Then for all $f \in C_c^{\infty}(M)$, $p\in (1,\infty)$, we have the following dimension-free estimate
\begin{equation} \label{E1}
	\| \vec{R}_\varphi^a f \|_{\textup{L}^p(T^*_xM)} \leqslant C_p\|f\|_{\textup{L}^p(M)},
\end{equation}
where $C_p>0$ is a constant depending only on $p$ proportional to $(p-1)^{-1}$ when $p\to 1$ and to $p$ when $p\to \infty$.
\end{theorem}

Our approach allows us to change the measure and also get dimensionless weighted estimates in $\textup{L}^p$ spaces for the Riesz vector on manifolds with non-negative curvature. For the Euclidean setting, we refer to \cite{DPW}. For the case of manifolds, such an estimate was previously only known for the case $p=2$ and it was also based on a form of a certain Bellman function (see \cite{D}). A priori the weight has to be globally in $\textup{L}^2$ so as to be able to define the \emph{flow characteristic}
\[\tilde{Q}_p(w)=\sup_{x,y}(Q(w))(x,y)(Q(w^{-\frac1{p-1}}))^{p-1}(x,y).\]
The collection of weights for which this characteristic is finite is denoted $\tilde{A}_p$.
There is also a natural way to extend the class of the weights to resemble more the classical case allowing local $L^1$ weights. In this case we require that constants are integrable in the manifold $M$ equipped with the measure $d\mu_{\varphi}$ so as to prove the theorem for cut weights, such as in \cite{D}, that are in $L^1 \cap L^{\infty} \cap \textup{L}^2$ and then define the characteristic by a limiting procedure and deduce the theorem. See \cite{D} for a detailed exposition in the case $p=2$.

\begin{theorem}[Weighted $\textup{L}^p$ estimate for the Riesz vector for non-negative curvature]\label{Riesz}
Suppose that $M$ is a complete Riemannian manifold without boundary and $\Ric_{\varphi}\geqslant 0$. Then for all $f \in C_c^{\infty}(M)$, $p\in (1,\infty)$, and $w \in \widetilde{A}_p$, we have the following dimension-free estimate
\begin{equation} \label{E1w}
	\| \vec{R}_\varphi f \|_{\textup{L}^p(T^*_xM, w)} \leqslant \tilde{C}_p\widetilde{Q}_p(w)^{\max\{1,\frac{1}{p-1}\}}\|f\|_{\textup{L}^p(M, w)},
\end{equation}
where $\tilde{C}_p>0$ is again a constant depending only on $p$ with the expected end point behavior.
\end{theorem}

Let us now discuss the results we obtain in probability theory.

We start by developing a sparse domination for a martingale $Y$, differentially subordinate to a martingale $X$. It relies on a known weak type estimate from \cite{Wang}.

\begin{theorem}[Sparse domination of $Y$]
\label{Theorem_sparsedomination}
Let $Y$ differentially subordinate to $X$ then there exists a sparse selection $$(X,Y) \mapsto \{ T^j\}_{j\geqslant 0}$$ such that almost surely
$$Y^*(\omega)\leqslant 8\sum_{j = 0}^{\infty} | X |_{T^j} (\omega)
    \chi_{E_j} (\omega),$$
    where the right hand side has finitely many terms almost surely.
\end{theorem}

Notice that there holds $\mathbb{P}(E_j)\to 0$ as $j\to\infty$. What follows below is a generalisation that will be of particular interest for the case of Riesz transforms:

First, we prove a weighted estimate for the sparse operator:
\begin{theorem}\label{Theorem_SX}
There exists $C_p\geqslant 0$ such that for all functions $X\in L^p(w)$ there holds
  \[ \| \mathcal{S}(X) \|_{L^p (w)} \leqslant C_{p} 
    Q_p (w)^{\max \{1, \frac1{p - 1}\}} \| X \|_{L^p (w)} . \]
  The estimate is sharp in terms of the dependence on $Q_p (w)$. 
\end{theorem}

\begin{corollary}\label{Theorem_YstarX}
  There exists $C_{p}\geqslant 0$ such that for all pairs $(X,Y)$ where $Y$ is differentially subordinate to $X$ there holds
  \[ \| Y^{\ast} \|_{L^p (w)} \leqslant C_{p} 
  Q_p (w)^{\max  \{1, \frac1{p - 1}\}} \| X \|_{L^p (w)} . \]
  The estimate is sharp in terms of the dependence on $Q_p (w)$.
\end{corollary}

In the special case $Y = X$ we also prove by a different method

\begin{theorem}
  \label{Theorem_XstarX}
  There exists $C_{p}\geqslant 0$ such that for all martingales $X$ there holds
   \[ \| X^{\ast} \|_{L^p (w)} \leqslant C_{p} Q_p (w)^{  \frac1{p - 1}} \| X \|_{L^p (w)} . \]
 The estimate is sharp in terms of the dependence on $Q_p (w)$. 
\end{theorem}

The weighted estimates are well known to be sharp in terms of the dependence on the
$A_p$ characteristic, already for dyadic filtration on $[0, 1]$ endowed with
Lebesgue measure. In this paper we focus on the upper estimates.

\medskip

For the remaining results, we will always make the following hypotheses

\begin{hypothesis} \label{H: hypotheses} 
Let $X$ and $Y$ be uniformly integrable c\`adl\`ag martingales so that $Y$ is differentially subordinate to $X$, and $X\geqslant 0$.  Let further $\widetilde{X}_t = X_t+A_t$ be a non-negative uniformly integrable submartingale where $A_t$ is a predictable increasing finite variation process. Finally, let $Z$ be a c\`adl\`ag semi-martingale such that $|Z_0|\leqslant |\widetilde{X}_0|$ and whose increments satisfy the stochastic differential equation 
\[\mathd Z_t=(V_{t_-}-a\Id)Z_{t_-} \mathd t +\mathd Y_t,\]
where $V_t$ is an adapted process with values in non-positive, symmetric $n \times n$ matrices and $a \geqslant 0$.
\end{hypothesis}

To treat the process $Z$, we begin with a weak type estimate targeted to $Z$, namely

\begin{lemma}[Weak-type estimate for $Z$]\label{L: weak type}
Let $\widetilde{X}_t$ and $Z_t$ as in Hypothesis \ref{H: hypotheses}.
Then for all $\lambda >0$ we have
\begin{equation}
\label{wte}
\mathbb{P}\left( (|Z_t|+|\widetilde{X}_t|)^*  \geqslant  \lambda \right) \leqslant 2 \lambda ^{-1} \|\widetilde{X}\|_1.
\end{equation}
\end{lemma}

The approach is similar to \cite{BanWan1995,Wang} regarding the treatment of jumps, and similar to \cite{BO} regarding the treatment of $Z$ in the case of continuous paths processes. However, we also seek an estimate with respect to $\widetilde{X}_t = X_t+A_t$ rather than $X_t$, and therefore need to adapt the proof to this situation. Remarkably, the same special function can be used for all those situations. We refer to Section \ref{SS: weak type Z} for the proof.

Next, we state our sparse domination. We recall that we denote by $Z^*=\sup_{t\geqslant 0}|Z_t|$ the maximal function associated with the process $Z$. 

\begin{theorem}[Sparse domination for $Z$]\label{T: sparse decomposition}
Let $\widetilde{X}_t$ and $Z_t$ as in Hypothesis \ref{H: hypotheses}.
Then there exists a sparse sequence $\{T^j\}_{j\geqslant 0}$ such that almost surely
\begin{equation}\label{sparse dec}
Z^*(\omega) \leqslant 8 S(\widetilde{X})(\omega),
\end{equation}
where $S(\widetilde{X})$ is the sparse operator associated to $\{T^j\}_{j\geqslant 0}$ defined by \eqref{sparse1}.
\end{theorem}

As a consequence of the sparse domination from Theorem \ref{T: sparse decomposition} and estimates for the sparse operator $S(\widetilde{X})$ we get the following respectively unweighted $L^p$ and weighted $\textup{L}^2$ estimates for $Z^*$.

\begin{theorem}[Estimates for $Z^*$] \label{Z est}
Let $\widetilde{X}_t$ and $Z_t$ as in Hypothesis \ref{H: hypotheses}.
Then for every $p>1$ there exists $C_{p}>0$ such that
\begin{equation}
\|Z^*\|_{\textup{L}^p}\leqslant C_{p} \| \widetilde{X} \|_{\textup{L}^p}.
\label{Zest_1}
\end{equation}
\end{theorem}  

\begin{theorem}[Weighted estimates for $Z^*$ for $p=2$] \label{Z w est}
Let $\widetilde{X}_t$ and $Z_t$ as in Hypothesis \ref{H: hypotheses}.
Let further $w$ be a positive uniformly integrable submartingale (a \emph{weight}).
Then
\begin{equation}
\|Z^*\|_{\textup{L}^2(w)}\leqslant {C}_2 Q_2(w)\| \widetilde{X} \|_{\textup{L}^2(w)}.
\label{Zest_2}
\end{equation}
\end{theorem}

In the above theorem, $Q_2(w)$ denotes the $A_2$ characteristic of the weight $w$ on a filtered space.

\section{Maximal Function of $X$}

In this section we prove Theorem \ref{Theorem_XstarX} via the modification of a simple and direct
 domination argument for the maximal function. See for example the argument by Lerner {\cite{Ler2008a}} in a different context. Notice that the obtained norm estimate is the same as that in 
 Buckley's text \cite{Buc1993a} on homogeneous spaces. Buckley's proof enjoyed an extension to some martingales with certain restrictive homogeneity 
 conditions in the presence of jumps - this was needed because of the failure of the openness of the $A_p$ class in the general context. The argument here 
 does not rely on the openness condition of the $A_p$ classes and is therefore providing the estimate in full generality and in addition recovers the correct growth with the $A_p$ characteristic for the norm estimate. The argument consists of a trajectory-wise domination of the maximal operator of $X$ and the use of Doob's inequality. We write $u=w^{\frac{1}{p-1}}$ for the dual weight and recall that the $A_p$ characteristic is $\sup_{\tau}\| u_{\tau}^{p-1}w_{\tau}\|_{\infty}$.
 We also introduce the notations 

 \begin{equation}\label{weighted_expectation} 
\mathbbm{E} [\,\cdot \; w ]
  =\mathbbm{E}_w [\,\cdot \,] \mathbbm{E} [w ] ,
   \end{equation}
  \begin{equation}\label{weighted_sampling} 
   \cdot_{\,\tau, w} =\mathbbm{E}_w [\, \cdot \;|\, 
\mathcal{F}_{\tau}],
  \end{equation}
\begin{equation}\label{notation_weightedmax} 
X_{ u}^{\ast}(\omega)=\sup_{t}| X_{t,u}(\omega)|.
\end{equation}

 There holds for all $t\geqslant 0$ and all $p>1$ 
\begin{eqnarray*}
  | X_t |^{p - 1} 
  & \leqslant &
   \left(\mathbbm{E} \left[| X |\, | \, \mathcal{F}_t\right]\right)^{p -1}\ =  
  \left(\mathbbm{E} \left[| X | u^{-1}u \, | \, \mathcal{F}_t\right]\right)^{p - 1}\\
  & = & 
  \left(\mathbbm{E}_u \left[| X | u^{-1} \, | \, \mathcal{F}_t\right]\right)^{p - 1} \left(\mathbbm{E}
  \left[u \, | \, \mathcal{F}_t\right]\right)^{p - 1}\\
  & \leqslant & 
  Q_p (w) \left(\mathbbm{E} \left[w  \, | \, \mathcal{F}_t\right]\right)^{- 1}
  \left(\mathbbm{E}_u \left[| X | u^{-1} \, | \,  \mathcal{F}_t\right]\right)^{p - 1}.
\end{eqnarray*}
Now observe that 
\begin{eqnarray*}
\left(\mathbbm{E}_u \left[| X | u^{-1}\, |\, \mathcal{F}_t\right]\right)^{p - 1}
&=&
\mathbbm{E} \left[\left(\mathbbm{E}_u [| X | u^{-1} \,  | \, \mathcal{F}_t]\right)^{p - 1} | \,\mathcal{F}_t\right] \\
&\leqslant& 
\mathbbm{E} \left[\left(\left(| X | u^{-1}\right)_{^{} u}^{\ast}\right)^{p - 1}w^{-1} w \, | \, \mathcal{F}_t\right]\\
&=&
\mathbbm{E}_w\left[ \left(\left(| X | w\right)_{^{} u}^{\ast}\right)^{p - 1}w^{-1}  \, | \, \mathcal{F}_t\right]  
\mathbbm{E}\left[ w \, | \, \mathcal{F}_t \right] .
\end{eqnarray*} 
Then get for all $t$
\[ | X_t |^{p - 1} \leqslant Q_p (w) \left(\left(\left(| X | u^{- 1}\right)^{\ast}_u\right)^{p - 1} w^{-1} \right)^{\ast}_w,\]
and therefore 
$$(X^{\ast })^p\leqslant \left( Q_p (w) \left(\left(\left(| X | u^{- 1}\right)^{\ast}_u\right)^{p - 1} w^{-1} \right)^{\ast}_w \right)^{\frac{p}{p-1}}.$$
Thus
\begin{eqnarray*}
  \mathbbm{E} [(X^{\ast})^p w] 
  & \leqslant & Q_p (w)^{\frac{p} {p - 1}} 
  \mathbbm{E}\left[\left(\left(\left(\left(| X | u^{- 1}\right)^{\ast}_u\right)^{p - 1} w^{- 1}\right)^{\ast}_w\right)^{p'} w\right]\\
  & = & Q_p (w)^{\frac{p} {p - 1}} 
  \mathbbm{E}_w \left[\left(\left(\left(\left(| X | u^{- 1}\right)^{\ast}_u\right)^{p -1} w^{- 1}\right)^{\ast}_w\right)^{p'}\right] \mathbbm{E} [w]\\
  & \leqslant & Q_p (w)^{\frac{p} {p - 1}}  \left( \frac{p'}{p' - 1} \right)^{p'}
  \mathbbm{E}_w \left[(((| X | u^{- 1})^{\ast}_u)^{p - 1} w^{- 1})^{p'}\right]
  \mathbbm{E} [w]\\
  & = & Q_p (w)^{\frac{p} {p - 1}}  \left( \frac{p'}{p' - 1} \right)^{p'} \mathbbm{E}
  \left[(((| X | u^{- 1})^{\ast}_u)^{p - 1} w^{- 1})^{p'} w\right]\\
  & = & Q_p (w)^{\frac{p} {p - 1}} \left( \frac{p'}{p' - 1} \right)^{p'} \mathbbm{E}
 \left[((| X | u^{- 1})^{\ast}_u)^p u\right]\\
  & = & Q_p (w)^{\frac{p} {p - 1}}  \left( \frac{p'}{p' - 1} \right)^{p'}
  \mathbbm{E}_u \left[((| X | u^{- 1})^{\ast}_u)^p\right] \mathbbm{E} [u]\\
  & \leqslant & Q_p (w)^{\frac{p} {p - 1}} \left( \frac{p'}{p' - 1} \right)^{p'}
  \left( \frac{p}{p - 1} \right)^p \mathbbm{E}_u \left[(| X | u^{- 1})^p\right]
  \mathbbm{E} [u]\\
  & = & Q_p (w)^{\frac{p} {p - 1}}  \left( \frac{p'}{p' - 1} \right)^{p'} \left(
  \frac{p}{p - 1} \right)^p \mathbbm{E} \left[(| X | u^{- 1})^p u\right]\\
  & = & Q_p (w)^{\frac{p} {p - 1}}  \left( \frac{p'}{p' - 1} \right)^{p'} \left(
  \frac{p}{p - 1} \right)^p \mathbbm{E} \left[| X |^p w\right].
\end{eqnarray*}
Raising to the power $1/p$ gives the desired estimate in Theorem \ref{Theorem_XstarX} with $C_{p}=\frac{p^{p'}}{p-1}$.

\

\section{Sparse domination of the process $Y$. }

In this section we prove Theorem \ref{Theorem_sparsedomination}.  Without loss of generality $X$ has non-zero closure and $\|X\|_1>0$. 
We will use the following preliminary weak type estimate due to Wang
{\cite{Wang}}.  

\begin{theorem}[Wang] Let $Y$ differentially subordinate to $X$ then for all $\lambda >0$ there holds
  \[ \mathbbm{P} (\{ \omega \in \Omega : (|Y|  + |X|)^{\ast}
     (\omega) \geqslant \lambda \}) \leqslant \frac2{\lambda} \| X \|_1 . \]
\end{theorem}
Notice that it trivially implies
\begin{corollary}[Wang]
Let $Y$ differentially subordinate to $X$ then  for all $\lambda >0$ there holds
  \[ \mathbbm{P} (\{ \omega \in \Omega : Y^{\ast} (\omega) \vee X^{\ast}
     (\omega) \geqslant \lambda \}) \leqslant \frac2{\lambda} \| X \|_1 . \]
\end{corollary}
It also implies the following.
\begin{lemma} \label{Theorem_WangWeakType}
Let $Y$ differentially subordinate to $X$ then for all $A\in \mathcal{F}_0$ and for all $\lambda >0$ there holds
  \[ \mathbbm{P} (\{ \omega \in A : Y^{\ast} (\omega) \vee X^{\ast}
     (\omega) > \lambda | X |_0(\omega) \}) \leqslant \frac2{\lambda} \mathbb{P}(A) . \]
\end{lemma}
\begin{proof}
Let us write 
$$\tilde{X} = \chi_{A \cap \{ |X|_0>0\}}X/|X|_0 \text{ and } \tilde{Y} = \chi_{A\cap \{|X|_0>0\}}Y/|X|_0.$$
Notice first that $|X|_0$ is measurable in $\mathcal{F}_0$. So the pair $(\tilde{X},\tilde{Y})$ are adapted martingales under differential subordination. 
Further, notice that almost surely $|X|_0(\omega) =0 \Rightarrow |X|_t(\omega) =0$ for all $t\geqslant 0$. In particular, all future 
increments of $X$ are zero and thus $|X_t(\omega)| \vee |Y_t(\omega)| =0$ for all  $t>0$. Thanks to this and Wang's Theorem we can estimate for all $\lambda >0$
\begin{align*} 
\mathbbm{P} (\{ \omega & \in A :
   Y_t^{ \ast} \vee X_t^{ \ast} > \lambda|X|_0 \})\\
 &= \mathbbm{P} (\{ \omega \in A\cap \{ |X|_0>0 \} : 
   Y_t^{ \ast} \vee X_t^{ \ast} > \lambda|X|_0 \}) \\
 &=  \mathbbm{P} (\{ \omega \in \Omega :
   \tilde{Y}_t^{ \ast} \vee \tilde{X}_t^{ \ast} > \lambda \})\\
& \leqslant  \frac{2}{\lambda} \| \tilde{X} \|_1  
=   \frac{2}{\lambda}  \mathbbm{E}  | \tilde{X} | \\
&= \frac{2}{\lambda}
   \mathbbm{E} \left[ \frac{\chi_{A\cap \{ |X|_0>0\}}}{| X |_0}\mathbbm{E} [| X | \, |\,  \mathcal{F}_0]
   \right] \\
&= \frac{2}{\lambda} \mathbb{P}(A). 
\end{align*}   

\end{proof}

\paragraph{Stopping procedure.}
The first step for the proof of Theorem  \ref{Theorem_sparsedomination}  is the stopping procedure. 
Let us select an increasing sequence $\{T^j\}_{j\ge0}$ of stopping times associated to the pair $X$ and $Y$ inductively. 
To do so, start with 
$$T^0(\omega)=\inf\{t>0: |X|_t(\omega)>0\},$$ that is, 
$T^0(\omega)= 0$ if $|X|_0(\omega) >0$ and $\infty$ otherwise. Set   
$$E_0=\{T^0<\infty \}=\{ \omega \in \Omega: |X|_0(\omega)>0\}.$$
Notice that since $\|X\|_1>0$ we have $\mathbb{P}(E_0)>0$ and  $|X|_{T_0}>0$ on $E_0$. 
Let us define $\mathfrak{F}^0=(\mathcal{F}_t^0)_{t\geqslant 0}$ by $\mathcal{F}^0_t =\mathcal{F}_{T^0 \vee t}$
and consider the martingales 
$$Y^0_t =\chi_{E_0}Y_t  \text{ and } X^0_t = \chi_{E_0}X_t .$$
   
Now we proceed with the iteration. Assume $n\geqslant 1$.
Assume we have filtrations $\mathfrak{F}^0,...,\mathfrak{F}^{n-1}$, an increasing sequence of 
stopping times $T^0,...,T^{n-1}$ with associated nested sets $E_0,..., E_{n-1}$ measurable in $\mathcal{F}_0,...,\mathcal{F}_{n-1}$ respectively and pairs of martingales 
$(X^0,Y^0),...,(X^{n-1},Y^{n-1})$ under differential subordination. 

We set the stopping time $T^n (\omega)$ by
\[ T^n (\omega) = \inf \{ t > 0 : Y^{n-1}_t  (\omega)\vee X^{n-1}_t (\omega)>4|X|_{T^{n-1}}(\omega) \}. \]
Notice that $|X|_{T^{n-1}}(\omega) =0 $ implies $|X|_t(\omega) =0$ for all $t\geqslant T^{n-1}$. In particular, all future 
increments of $X$ are zero and thus $|X^{n-1}_t(\omega)| \vee |Y^{n-1}_t(\omega)| =0$ for all $t>T^{n-1}$. There holds thus $T^n(\omega)= \infty$ if $|X|_{T^{n-1}}(\omega) =0$, and
\[ E_n = \{ T^n<\infty \}=\{ \omega \in \Omega : Y^{n-1 \ast}(\omega) \vee X^{n-1 \ast}(\omega) > 4|X|_{T^{n-1}}(\omega) \} .\]
 Let
the filtration $\mathfrak{F}^n = (\mathcal{F}_t^n)_{t \geqslant 0} :=
(\mathcal{F}_{T^n \vee t})_{t \geqslant 0}$. Observe that $E_n \in
\mathcal{F}_{T^n}=\mathcal{F}_0^n$.

\paragraph{Choice of the iterates.} We now take care of the foot of the next pair of martingales, particularly in the presence of a jump at a stopping time. 
By differential subordination we can compare the jumps almost surely 
$$
|\Delta Y^{n-1}_{T^n}| := \left| Y^{n-1}_{T^n} (\omega) - Y^{n-1}_{T^n_-}(\omega) \right| \leqslant \left| X^{n-1}_{T^n} = (\omega) - X^{n-1}_{T^n_-} (\omega) \right| =: |\Delta X^{n-1}_{T^n}|.$$ 
Thus, there exists a linear operator $r_{T^n} (\omega) \in \mathcal{F}_{T^n}$ such that $| r_{T^n} |
\leqslant 1$ and 
$$Y^{n-1}_{T^n} (\omega) - Y^{n-1}_{T^n_-} (\omega) = r_{T^n} (\omega)  \left(X^{n-1}_{T^n}(\omega) - X^{n-1}_{T^n_-} (\omega)\right).$$ 

In particular, one sets $r_{T^n} =0$ if $\Delta X^{n-1}_{T^n}=0$. Define now the new iterates $(X_t^n,Y_t^n)$ given $(Y_t^{n-1},Y_t^{n-1})$ as
\begin{eqnarray*}
Y_t^n 
&=& \chi_{E_n} \left(\mathbbm{E} [Y^{n-1} \, |\, \mathcal{F}^n_t] - Y^{n-1}_{T^n} + r_{T^n} X^{n-1}_{T^n}\right)  \\
&=& \chi_{E_n} \left( r_{T^n} X^{n-1}_{T^n} + \int^{t \vee T^n}_{T^n_+} \mathd Y^{n-1}_u \right) 
\end{eqnarray*}
and
\begin{eqnarray*}
 X_t^n &=& \chi_{E_n} \mathbbm{E} [X^{n-1} \; | \; \mathcal{F}^n_t]  \\
 &=& \chi_{E_n}  \left( X^{n-1}_{T^n} + \int^{t \vee T^n}_{T^n_+} \mathd X^{n-1}_u \right).
\end{eqnarray*}
By induction, these are martingales with respect to $\mathfrak{F}^n$ and $Y^n$ is
differentially subordinate to $X^n$.

\paragraph{Sparsity.} We prove that the resulting selection $\{T^ j\}_{j\geqslant 0}$ is sparse. Let $n\geqslant 0$. 
Let $A^{n}\subset E_{n}$ with $A^{n} \in \mathcal{F}_{T^{n}}$. Then $\chi_{A^{n}}Y^{n}_t $ is differentially 
subordinate to $\chi_{A^{n}}X^{n}_t $ in $\mathfrak{F}^{n}$.  Notice that $|X|_{T^{n}}=|X^{n}|_{0}$ since $\mathcal{F}^{n}_0=\mathcal{F}_{T^{n}}$.

Thanks to Lemma \ref{Theorem_WangWeakType} applied to $\mathcal{F}=\mathcal{F}^{n}$, $X_t=\chi_{E_n} X_t^n$, $Y_t=\chi_{E_n} Y_t^n$, and $A=A^n$  there holds
$$\mathbbm{P} (A^{n} \cap E_{n+1})  = \{ \omega \in A^n : Y^{n \ast}(\omega) \vee X^{n \ast}(\omega) > 4 |X|_{T^n}(\omega)\} \leqslant  \frac{1}{2} \mathbb{P}(A^{n}).$$

\medskip

\paragraph{Domination.} We now prove the domination estimate. Indeed, we show that for all $n\geqslant 0$ there holds almost surely
\begin{equation}\label{estimate_inductive} Y^{\ast}\leqslant \sum_{j=0}^{n-1}8|X|_{T^j}\chi_{E_j}+Y^{n \ast}\quad (\mathcal{E}_n).\end{equation}
This implies the required domination because for the support of $Y^{n \ast}$ there holds
  $$\mathbb{P}\left(\text{supp} (Y^{n \ast})) \leqslant \mathbb{P}(E_{n}\right) \to 0 \text{ as } n\to \infty.$$
For $n=0$, the estimate $\mathcal{E}_0$ in (\ref{estimate_inductive}) follows from $$Y^{\ast}=Y^{\ast}\chi_{\Omega \setminus E_0} +  Y^{\ast}\chi_{E_0} = 0+ Y^{0 \ast}.$$
Assuming now $(\mathcal{E}_n)$ holds in (\ref{estimate_inductive}), we pass to $(\mathcal{E}_{n+1})$. Since $Y^{n \ast}$ is supported on $E_n$ we split 
$$Y^{n \ast} =  Y^{n \ast}\chi_{E_n \setminus E_{n+1}}+   Y^{n \ast}\chi_{E_{n+1}}.$$ 
In the complement of $E_{n+1}$ we have $Y^{n \ast} \leqslant 4 |X|_{T^n}\chi_{E_n}\leqslant 8 |X|_{T^n}\chi_{E_n}+Y^{n+1 \ast}$. In $E_{n+1}$ we have 
$$Y^{n \ast} (\omega)=\max \left\{ \sup_{t<T^{n+1}(\omega)} |Y^n_t(\omega)|, \sup_{t\geqslant T^{n+1}(\omega)} |Y^n_{t}(\omega)|   \right\}.$$
The first supremum is bounded by $4|X|_{T^n}\chi_{E_n}$ and for the second supremum we write trajectory-wise for $t\geqslant T^{n+1}(\omega)$
\begin{align*}
Y^n_t &=  Y^n_0 + \int^{T^{n+1}}_{0_+} \mathd Y^n_u + \int^t_{T^{n+1}_+} \mathd Y^n_u  \\
 &= Y^n_0 + \int^{T^{n+1}_-}_{0_+} \mathd Y^n_u 
+  (Y_{T^{n+1}}-Y_{T^{n+1}_-})  + \int^t_{T^{n+1}_+} \mathd Y^n_u  \\
&=    \left( Y^n_0 + \int^{T^{n+1}_-}_{0_+} \mathd Y^n_u \right) - (r_{T^{n+1}} X^n_{T^{n+1}_-}) +
   \left( r_{T^{n+1}}  X^n_{T^{n+1}} + \int^t_{T^{n+1}_+} \mathd Y^n_u \right).
\end{align*}
We estimate in $E_{n+1}$ for $t\geqslant T^{n+1}$  
$$|Y^ n_t|  \leqslant  \left|Y^n_0 + \int^{T^{n+1}_-}_{0_+} \mathd Y^n_u \right| + | X^n_{T^{n+1}_-}| +  \left| r_{T^{n+1}} X_{T^{n+1}} + \int^t_{T^{n+1}_+} \mathd Y^n_u \right|.$$
      
The first two summands are each controlled by $4 | X |_{T^n}\chi_{E_n}$ by the definition of
the stopping time $T^{n+1}$. Last, observe that the third term on $E_{n+1}$ is
dominated by $ Y^{n+1 \ast}$.

Gathering the information, there holds almost surely
\[ Y^{n \ast} \leqslant 8 \, | X |_{T^n}\chi_{E_n} + Y^{n+1 \ast} , \] 
and thus
$$Y^{\ast} \leqslant \sum_{j=0}^{n-1} 8\,|X|_{T^j}\chi_{E_j} + Y^{n \ast} \leqslant \sum_{j=0}^{n} 8\,|X|_{T^j}\chi_{E_j} +Y^{n+1 \ast} .$$
The claim $(\mathcal{E}_{n+1})$ in (\ref{estimate_inductive}) is proved and the sparse domination in Theorem \ref{Theorem_sparsedomination} follows. 

\

\subsection{Maximal function of $Y$}

In this section we prove Theorem \ref{Theorem_SX} and Theorem \ref{Theorem_YstarX}. We first prove Theorem \ref{Theorem_SX} for $p=2$ and then obtain the result for other values of $p$ via extrapolation. We then deduce Theorem \ref{Theorem_YstarX} via the sparse domination from Theorem \ref{Theorem_sparsedomination}. 

In order to prove Theorem \ref{Theorem_SX} 
for the case $p = 2$, it suffices to show that there exists $C_{2}\geqslant 0$ such that for all $w \in A_2$ and all functions $X \in L^2(w)$ there holds
\begin{equation} \label{sparse_weighted_estimate}
\| \mathcal{S} ( {X} ) \|_{L^2 (w)} \leqslant C_{2} Q_2 (w) \|
   {X} \|_{L^2 (w)} . 
\end{equation}   
This means 
\[ \left(\mathbbm{E} [(\mathcal{S} ( {X} ))^2 w]\right)^{\frac12} \leqslant C_{2} Q_2
   (w)\left( \mathbbm{E} [| {X} |^2 w] \right)^{\frac12}. \]
Dualizing and writing $u=w^{-1}$, we reduce to the estimate
\[ \mathbbm{E} [\mathcal{S} ({X} ) | {Z} |] \leqslant
   C_{2} Q_2 (w) \mathbbm{E} [| {X} |^2 w]^{\frac12} \mathbbm{E} [|{Z} |^2 u]^{\frac12}. \]
We recall the notations (\ref{weighted_expectation}), (\ref{weighted_sampling}) 
\begin{equation*} 
\mathbbm{E} [\,\cdot \; w ]
  =\mathbbm{E}_w [\,\cdot \,] \mathbbm{E} [w ] ,
   \end{equation*}
  \begin{equation*}  
   \cdot_{\,\tau, w} =\mathbbm{E}_w [\, \cdot \;|\, 
\mathcal{F}_{\tau}].
  \end{equation*}

Then, we write
$| \tilde{X} | u = | {X} |$ and $| \tilde{Z} | w = | {Z} |$ then suppressing the $\tilde \cdot$ again, it suffices
to prove
\[ \mathbbm{E} [\mathcal{S} (| X | u) | Z | w] 
\leqslant C_{2} Q_2 (w)
   \mathbbm{E} [w]^{\frac12} \mathbbm{E} [u]^{\frac12} 
   \mathbbm{E}_u [| X |^2]^{\frac12} \mathbbm{E}_w [| Z |^2]^{\frac12} . \]
Now, we calculate the left hand side
\begin{eqnarray*}
 \mathbbm{E} \left[ \sum_j (| X | u)_{T^j} \chi_{E_j} | Z | w \right]
   &=&\mathbbm{E} \left[ \sum_j \mathbbm{E} [(| X | u)_{T^j} \chi_{E_j} | Z | w \;
   | \; \mathcal{F}_{T^j}] \right] \\
   &=&\mathbbm{E} \left[ \sum_j (| X | u)_{T^j} (|
   Z | w)_{T^j} \chi_{E_j} \right] \\
  &\leqslant& Q_2 (w) \mathbbm{E} \left[ \sum_j | X |_{T^j, u} | Z |_{T^j, w} \chi_{E_j}
   \right]. 
 \end{eqnarray*}  
In the above calculation, we used the notations (\ref{weighted_expectation}) and (\ref{weighted_sampling}) and noticed that
\begin{equation*}
(| Z | w)_{T^j} \chi_{E_j} =\mathbbm{E} [| Z | w\, |\, \mathcal{F}_{T^j}]
   \chi_{E_j} =\mathbbm{E}_w [| Z | | \mathcal{F}_{T^j}] \mathbbm{E} [w \,|\,
   \mathcal{F}_{T^j}] \chi_{E_j} ,
   \end{equation*}
and similarly for the other term. We recalled that by the $A_2$ condition
\[ \| \mathbbm{E} [w \,|\, \mathcal{F}_{T^j}] \mathbbm{E} [w^{- 1} \,|\,
   \mathcal{F}_{T^j}] \|_{\infty} \leqslant Q_2 (w). \]
For each fixed $j$ we have that the non-negative random variable $$| X
|_{T^j, u} | Z |_{T^j, w} \chi_{E_j} \in \mathcal{F}_{T^j}$$ and as such it can
be approximated from below by step functions.
\[ \sum_k \alpha_k^j \chi_{A_k^j} \nearrow | X |_{T^j, u} | Z |_{T^j, w}
   \chi_{E_j} ,\]
with $A^j_k \in \mathcal{F}_{T^j}$ disjoint and $\stackrel{\cdot}{\cup}_k A^j_k = E_j$. Notice
that on $A^j_k$ there holds
\begin{equation}\label{estimate_maximal}
 \alpha^j_k \chi_{A^j_k}(\omega)\leqslant X_{u}^{\ast} Z_{ w}^{\ast} (\omega),
 \end{equation}
where the maximal functions are taken with respect to 
 weighted measure:
\begin{equation*} 
X_{ u}^{\ast}(\omega)=\sup_{t}| X_{t,u}(\omega)|.
\end{equation*}

\

Now recall that $\mathbbm{P} (A_k^j \cap E_{j + 1}) \leqslant \frac{1}{2}
\mathbbm{P} (A_k^j)$ and so if we write  $$S^{j+1}_k = A_k^j \backslash (A_k^j \cap E_{j + 1}),$$ then
$\mathbbm{P} (S_k^{j+1} ) \geqslant \frac{1}{2} \mathbbm{P} (A_k^j)$. 
We changed the index on $S$ to recall the important fact that it is measurable in $\mathcal{F}_{T^{j+1}}$. 
Notice the crucial property of the collection $\{S_k^{j+1}\}_{j,k\geqslant 0}$ : it is a disjoint collection in both parameters.

\

We estimate 
\begin{eqnarray*}
  &&\mathbbm{E} \left[ \sum^J_{j = 0} \sum_k \alpha_k^j \chi_{A_k^j} \right] 
  =  
  \sum^J_{j = 0} \sum_k \alpha_k^j \mathbbm{E} \left[\chi_{A^j_k}\right]
   =  \sum^J_{j = 0} \sum_k \alpha_k^j \mathbbm{P} \left(A^j_k\right)\\
  & \leqslant & 
  2 \sum^J_{j = 0} \sum_k \alpha^j_k \mathbbm{P} \left(S^{j+1}_k\right)
   =  2\,\mathbbm{E} \left[ \sum^J_{j = 0} \sum_k \alpha_k^j \chi_{S^{j+1}_k} \right]\\
  & = & 
  2\,\mathbbm{E} \left[ \sum^J_{j = 0} \sum_k \alpha_k^j\chi_{S^{j+1}_k} w^{\frac12} u^{\frac12} \right]
  \leqslant  
  \mathbbm{E} \left[ \sum^J_{j = 0} \sum_k X^{\ast}_{ u} u^{\frac12} Z^{\ast}_{ w} w^{\frac12} \chi_{S^{j+1}_k}  \right]\\
   &\leqslant&  
   2 \left( \mathbbm{E} \left[ \sum^J_{j = 0} \sum_k (X^{\ast}_{ u})^2 u \chi_{S^{j+1}_k} \right] \right)^{\frac12} 
    \left(\mathbbm{E} \left[ \sum^J_{j = 0} \sum_k (Z^{\ast}_{ w})^2 w \chi_{S^{j+1}_k} \right] \right)^{\frac12}\\[.5em] 
  & \leqslant & 
  2 \left(\mathbbm{E} \left[\left(X^{\ast}_u\right)^2 u\right]\right)^{\frac12} \left(\mathbbm{E}
  \left[\left(Z_w^{\ast}\right)^2 w\right]\right)^{\frac12}\\[.5em] 
  & = & 
  2 \left(\mathbbm{E} [u]\right)^{\frac12} \left(\mathbbm{E} [w]\right)^{\frac12} \left(\mathbbm{E}_u
  \left[\left(X^{\ast}_u\right)^2\right]\right)^{\frac12} \left(\mathbbm{E}_w \left[\left(Z^{\ast}_w\right)^2\right]\right)^{\frac12}\\[.5em] 
  & \leqslant & 
  8 \,(\mathbbm{E} [w])^{\frac12} (\mathbbm{E} [u])^{\frac12}
  \left(\mathbbm{E}_u \left[| X |^2\right]\right)^{\frac12} \left(\mathbbm{E}_w \left[| Z |^2\right]\right)^{\frac12}.
\end{eqnarray*}
By the monotone convergence theorem, this gives us the estimate
\[ \mathbbm{E} \left[ \sum_j | X |_{T^j, u} | Z |_{T^j, w} \chi_{E_j} \right]
   \leqslant 8\,\mathbbm{E} [w]^{\frac12} \mathbbm{E} [u]^{\frac12} \mathbbm{E}_u \left[|
   X |^2\right]^{\frac12} \mathbbm{E}_w \left[| Z |^2\right]^{\frac12}, \]
   and  we have thus seen that inequality (\ref{sparse_weighted_estimate}) holds with $C_{2}=8$.

   \
   
We now point out the changes for the case $p \neq 2$. We use the extrapolation theorem below from \cite{DPW}. 
Notice that only the weights are required to have the martingale property while $X,Y$ are functions.

\begin{theorem}[Domelevo-Petermichl]\label{thmextrap}
  Given a filtered probability space as described above. Let $1 < p < \infty$
  and $w \in A_p$. Let  $X, Y \in L^p(w)$. Suppose
  $1 < r < \infty$ and suppose $\forall A \geqslant 1 \; \exists N_r (A) > 0$
  increasing such that for triples $X, Y, \rho$ with $X,Y \in L^r(\rho)$ and
  $Q_r (\rho) \leqslant A$ there holds
  \begin{eqnarray*}
    \| Y \|_{L^r(\rho)} \leqslant N_r (A) \| X \|_{L^r(\rho)} .
  \end{eqnarray*}
  Then for any $1 < p < \infty$ there exists $N_p (B) > 0$ such that if
  $Q_p (w)  \leqslant B$ there holds
  \begin{eqnarray*}
    \| Y \|_{L^p(w)} \leqslant N_p (B) \| X \|_{L^p(w)} .
  \end{eqnarray*}
  Let $c_{\ref{Theorem_XstarX},p}$ denote the numeric part of the estimate in the
  weighted $L^p$ maximal estimate  from Theorem \ref{Theorem_XstarX}. We have 
   $$N_p (B) \leqslant 2^{\frac1{r}} N_r \left(2 c_{\ref{Theorem_XstarX},p'}^{\frac{p - r}{p - 1}} B\right) \text{ if } p > r .$$ 
  $$N_p (B) \leqslant 2^{\frac{r - 1}{r}} N_r \left( 2^{r - 1} \left(c_{\ref{Theorem_XstarX},p}^{p - r} B\right)^{\frac{r - 1}{p - 1}} \right) \text{ if } p < r. $$
\end{theorem}

Using this theorem for $r=2$ we extrapolate from inequality (\ref{sparse_weighted_estimate}) 
\[ \| \mathcal{S} ( {X} ) \|_{L^2 (w)} \leqslant 8\, Q_2 (w) \| {X} \|_{L^2 (w)} , \]
so $N_2(A)=8A$. We obtain  the estimate claimed in Theorem \ref{Theorem_SX}
$$\| \mathcal{S} ( X ) \|_{L^p (w)} \leqslant C_{p} Q_p (w)^{\max \{1, \frac1{p-1}\}} \| X \|_{L^p (w)}.$$ 
Finally 
$$\| Y^{\ast} \|_{L^p (w)} \leqslant 8\| \mathcal{S} ( X ) \|_{L^p (w)}$$ 
gives the claimed estimate in Theorem \ref{Theorem_YstarX}.

\section{The stochastic process $Z$}
In this section, we prove Lemma \ref{L: weak type}, Theorem \ref{T: sparse decomposition} as well as Theorems \ref{Z est} \& \ref{Z w est}. Recall that the different processes involved obey hypothesis \ref{H: hypotheses} and without loss of generality, we assume that the nonnegative process $\widetilde{X}$ has non-zero closure.

\subsection{Weak type estimate relevant to $Z$} \label{SS: weak type Z}

In this section we will prove the weak type estimate, Lemma \ref{L: weak type}.

\begin{proof}[Proof of Lemma \ref{L: weak type}]
This proof is modelled after the exposition in   Ba\~nuelos-Osekowski \cite{BO}. We extend it to tackle the presence of jumps as well as the use of the nonnegative submartingale $\widetilde{X}$ instead of the martingale $X$. We aim to show \eqref{wte}:
\begin{equation*} 
\mathbb{P}\left( (|Z_t|+|\widetilde{X}_t|)^*  \geqslant \lambda \right) \leqslant 2 \lambda ^{-1} \|\widetilde{X}\|_1.
\end{equation*}
It suffices to show the inequality for $\lambda =1$. Indeed, we can replace $\widetilde{X}_t$ and $Z_t$ in the above inequality with $\lambda^{-1} \widetilde{X}_t$ and $\lambda^{-1}Z_t$ to get \eqref{wte}.

To  show the inequality for $\lambda =1$, define functions $V,U:\mathbb{R}\times \mathbb{R}^n \to \mathbb{R}$ by 
$$
V(x,y) = \left\{
    \begin{array}{ll}
        -2|x| & \mbox{when } |x|+|y|<1, \\
        1-2|x| & \mbox{when } |x|+|y| \geqslant1,
    \end{array}
\right.
$$
and
$$
U(x,y) = \left\{
    \begin{array}{ll}
        |y|^2-|x|^2 & \mbox{when } |x|+|y|<1, \\
        1-2|x| & \mbox{when } |x|+|y| \geqslant1.
    \end{array}
\right.
$$

Let us first observe that everywhere $V\leqslant U$. When $|x|+|y|\geqslant 1$ this is obvious. In the case $|x|+|y|< 1$ we see that $V\leqslant U \Leftrightarrow -2|x| \leqslant |y|^2-|x|^2 \Leftrightarrow (|x|-1)^2\leqslant |y|^2+1$ but with $0\leqslant |x|<1$ and $|y|\geqslant 0$ we have $(|x|-1)^2\leqslant 1$ and $|y|^2+1\geqslant 1$. 

Define the stopping time 
\[T=\inf \{t \geqslant0: |\widetilde{X}_t|+|Z_t| \geqslant1\}\]
so that $|\widetilde{X}_T|+|Z_T| \geqslant1$ and $ |\widetilde{X}_t|+|Z_t| < 1$ for $t<T$. We will prove that
\begin{equation}\label{goal}\E U(\widetilde{X}_T,Z_T) \leqslant 0.\end{equation}
Here is how we deduce the weak-type estimate \eqref{wte} from (\ref{goal}). We have $V\leqslant U$ everywhere and $\E U(\widetilde{X}_T,Z_T)\leqslant 0$. Therefore
\begin{align*}
0&\geqslant\E U(\widetilde{X}_T,Z_T) \\
&\geqslant\E V(\widetilde{X}_T,Z_T) \\
&=\E (V(\widetilde{X}_T,Z_T)\chi_{\{|\widetilde{X}_T|+|Z_T|\geqslant1\}})+\E (V(\widetilde{X}_T,Z_T)\chi_{\{|\widetilde{X}_T|+|Z_T|< 1\}}) \\
&=\E ((1-2|\widetilde{X}_T|)\chi_{\{|\widetilde{X}_T|+|Z_T|\geqslant1\}})+\E ((-2|\widetilde{X}_T|)\chi_{\{|\widetilde{X}_T|+|Z_T|<1\}}) \\
&=\mathbb{P}(|\widetilde{X}_T|+|Z_T|\geqslant 1)-2\E|\widetilde{X}_T|,
\end{align*}
from which we deduce 
\[\mathbb{P}((|\widetilde{X}_t|+|Z_t|)^*\geqslant1) \leqslant 2\|\widetilde{X}\|_1,\]
which is exactly \eqref{wte} for $\lambda=1$.

In order to prove (\ref{goal}), we will split as follows
\[\E U(\widetilde{X}_T,Z_T)=\E(U(\widetilde{X}_T,Z_T)\chi_{\{T=0\}})+\E(U(\widetilde{X}_T,Z_T)\chi_{\{T>0\}})\]
and show that both of these contributions are non-positive.

\paragraph{Case $T=0$.}
First we handle the case when $T=0$. For such $\omega$ where $T=0$ by definition of $T$ we have $|\widetilde{X}_0|+|Z_0| \geqslant1$ and $U(\widetilde{X}_0,Z_0)=1-2|\widetilde{X}_0|$ and so since $|Z_0| \leqslant |\widetilde{X}_0|$, we have $1\leqslant |\widetilde{X}_0|+|Z_0|\leqslant 2 |\widetilde{X}_0|$ so that $1-2|\widetilde{X}_0| \leqslant 0$. For $|\widetilde{X}_0|+|Z_0| \geqslant1$ we have thus $U(\widetilde{X}_0,Z_0)=1-2|\widetilde{X}_0| \leqslant 0$ and so
\[\E (U(\widetilde{X}_T,Z_T)\chi_{\{T=0\}})=\E (U(\widetilde{X}_0,Z_0)\chi_{\{T=0\}}) \leqslant 0.\]

\paragraph{Case $T>0$.}
Now we take care of the $\omega$ where $T>0$. By simple calculations on the derivatives of $U$, we check that for $|x|+|y| < 1$ we have
\begin{align}
\partial_{x} U(x,y)&= -2x, \label{signU1}\\
\partial_{y_i} U(x,y)&= 2y_i, \label{signU2}\\
\partial^2_{xx} U(x,y)&= -2  \label{signU3},\\
\partial^2_{xy_i} U(x,y)&= 0, \label{signU4}\\
\partial^2_{y_iy_j} U(x,y)&= 2\delta_{ij} \label{signU5},
\end{align}
where $\delta_{ij}$ is the Kronecker delta symbol.
On $\{T >0\}$, the process $(\widetilde{X}_t,Z_t)$ evolves in the set $\{(x,y): \ |x|+|y|<1\}$, in the interior of which the function $U$ is twice differentiable, which means that we have the following It\^o formula
\begin{equation}
U(\widetilde{X}_T,Z_T)=I_0 + I_1+  I_2 + I_3,
\label{ito}
\end{equation}
with
\begin{align*}
I_0 & = U(\widetilde{X}_0,Z_0) \\
I_1 &= \int_{0_+}^T \partial_x U(\widetilde{X}_s,Z_s)\mathd \widetilde{X}_s + \sum_{i}\int_{0_+}^T\partial_{y_i} U(\widetilde{X}_s,Z_s)\mathd Z^i_s \\
I_2 &=\frac{1}{2} \int_{0_+}^T \partial^2_{xx} U(\widetilde{X}_s,Z_s) \mathd [ \widetilde{X},\widetilde{X} ]_s +  \sum_{i}\int_{0_+}^T\partial^2_{xy_i} U(\widetilde{X}_s,Z_s) \mathd [ \widetilde{X}, Z^i ]_s \\
& \qquad\quad + \frac{1}{2} \sum_{i}\sum_{j}\int_{0_+}^T \partial^2_{y_iy_j} U(\widetilde{X}_s,Z_s) \mathd [ Z^i,Z^j ] _s, \\
 I_3 = & \sum_{0<s\leqslant T}    \Big(  U(\widetilde{X}_s,Z_s)-U(\widetilde{X}_{s_-},Z_{s_-}) 
 		- \partial_x U(\widetilde{X}_{s_-},Z_{s_-})\Delta X_s \\
& \qquad\quad - \sum_{i}\partial_{y_i} U(\widetilde{X}_{s_-},Z_{s_-})\Delta Y^i_s \\
& \qquad\quad - \frac{1}{2} \partial^2_{xx} U(\widetilde{X}_{s_-},Z_{s_-})   | \Delta\widetilde{X}_s |^2
 - \sum_{i}\partial^2_{xy_i} U(\widetilde{X}_{s_-},Z_{s_-}) \Delta\widetilde{X}_s \Delta Z^i_s \\
& \qquad\quad
		-  \frac{1}{2} \sum_{i}\sum_{j} \partial^2_{y_i y_j} U(\widetilde{X}_{s_-},Z_{s_-}) \Delta Z^i_s \Delta Z^j_s
		  \Big).
\end{align*}
where the last sum is over jumping times.

\paragraph{The term $I_0$.}
Finally, since $T>0$ we have $U(\widetilde{X}_0,Z_0)=|Z_0|^2-|\widetilde{X}_0|^2 \leqslant 0$ and taking expectation in \eqref{ito} implies
\[\E (U(\widetilde{X}_T,Z_T)\chi_{\{T>0\}}) \leqslant 0.\]

\paragraph{The term $I_1$.}
First we study $I_1$. Recall that $Z_t$ satisfies the following stochastic differential equation
\begin{equation}\label{dZ}
\mathd Z_t=(V_{t_-}-a\Id)Z_{t_-} \mathd t+\mathd Y_t
\end{equation}
and that $\widetilde{X}_t=X_t+A_t$ which implies 
\[\mathd \widetilde{X}_t=\mathd X_t+\mathd A_t,\] 
where $(X_t)_{t\geqslant 0}$ is a martingale, and $(A_t)_{t\geqslant 0}$ is an increasing FV process. Now if we replace $\mathd \widetilde{X}_s$ and $\mathd Z_s$ in the expression of $I_1$ by the above formulas, we will obtain a local martingale part which is 
\[\int_{0_+}^T \partial_x U(\widetilde{X}_s,Z_s)\mathd X_s + \int_{0_+}^T\langle \partial_y U(\widetilde{X}_s,Z_s), \mathd Y_s\rangle\]
and a process 

\[ N_T =\int_{0_+}^T \partial_x U(\widetilde{X}_s,Z_s)\mathd A_s +\int_{0_+}^T\langle \partial_y U(\widetilde{X}_s,Z_s), (V_s-a\Id)Z_s\rangle \mathd s. \]

We may assume that the local martingale is a true martingale without loss of generality and hence its expectation is null. As for the process $N_T$, by \eqref{signU1} and \eqref{signU2} we have 
\[N_T=-2\int_{0_+}^T \widetilde{X}_s\mathd A_s+2\int_{0_+}^T \langle Z_s,(V_s-a\Id)Z_s\rangle \mathd s.\]

The non-positivity of the first term follows from $\mathd A_s \geqslant0$, since $(A_t)$ is an increasing FV process, and the non-negativity of $(\widetilde{X}_t)_{t\geqslant 0}$. The second term is non-positive because $V-a\Id$ takes values in the class of non-positive matrices, so we have $\langle Z_s,(V_s-a\Id)Z_s\rangle \leqslant 0$.  In conclusion, $N_T\leqslant 0$ and taking the expectation of $I_1$ gives us

\[\E(I_1 \chi_{T>0}) \leqslant 0.\]

\paragraph{The term $I_2$.}
Now we deal with $I_2$. By \eqref{signU3}-\eqref{signU5}, we obtain that
\[ I_2 = ( \langle Z,Z \rangle_T -|Z_0|^2-  \langle \widetilde{X},\widetilde{X} \rangle_T + |\widetilde{X}_0|^2 ).\]
Since $Z_t$ satisfies \eqref{dZ}, by integrating we get
\[Z_t - Z_0 =\int_{0_+}^t(V_s-a\Id)Z_s\mathd s + Y_t - Y_0.\]
Taking the quadratic covariation on both sides we obtain 
\begin{align*} 
\langle Z,Z \rangle_t - |Z_0|^2 &= \langle Y,Y \rangle_t - |Y_0|^2, \ \ \ \forall t\geqslant0 \\
& \leqslant  \langle X,X \rangle_t - |X_0|^2 \ \text{ by differential subordination.}
\end{align*}
In particular
$$ \E (\langle Z,Z \rangle_t - |Z_0|^2)  \leqslant  \E (\langle X,X \rangle_t - |X_0|^2) $$
On the other hand, Since $\widetilde{X}_t=X_t+A_t$, where $(A_t)_{t\geqslant 0}$ is a predictable FV process we have for all $t\geqslant 0$ 
\begin{equation*}
\E \langle \widetilde{X},\widetilde{X} \rangle_t = \E [X,X]_t + 2 \E [X,A]_t + \E[A,A]_t =  \E [X,X]_t + \E[A,A]_t
\end{equation*}
and therefore, since $A_t$ is increasing,
\begin{align*} 
 \E (\langle \widetilde{X},\widetilde{X} \rangle_t - \E \langle \widetilde{X},\widetilde{X} \rangle_0)
& = \E ([X,X]_t - [X,X]_0) + \E ([A,A]_t - [A,A]_0)\\
& \geqslant \E ([X,X]_t - [X,X]_0). 
\end{align*}

It follows immediately that\[\E (I_2 \chi_{T>0} ) \leqslant 0.\]

\paragraph{The term $I_3$.}
For $I_3$, the additional properties of the function $U$ that are required are these: If $|x|+|y|<1$ then 
 \begin{equation}\label{jumpconvexity}
 U(x+h,y+k)-U(x,y)-\partial_x U(x,y)h -\sum_j \partial_{y_j}U(x,y)k_j\leqslant -h^2 + |k|^2.
 \end{equation}
This inequality is true irrespective of whether $|x+h|+|y+k|<1$ or $|x+h|+|y+k|\geqslant 1$. In the first case, which means the process is still inside the diamond after the jump, there is an equality -- a simple consequence of the expression of $U$ and the definition of quadratic covariation. It therefore turns out that the only possible contribution in the sum defining $I_3$ is at the stopping time $s=T$. This is the second case $|x+h|+|y+k|\geqslant 1$ meaning the jump moves the process outside of the diamond. Now, if $h$ and $k$ are the jump increments of $\widetilde{X}$ and $Z$ respectively and observing that $\Delta Z_s=\Delta Y_s$ at jump times, we have
\begin{equation}
	\E |\Delta Z_s|^2 = \E |\Delta Y_s|^2 \leqslant \E |\Delta X_s|^2 \leqslant \E |\Delta \widetilde{X}_s|^2
\end{equation}
where we used successively that $Y$ is differentially subordinate to $X$ and that $\E \Delta X_s \Delta A_s = 0$.
As a conclusion, we have again $ \E I_3 \leqslant 0$.

This concludes the proof of  Lemma \ref{L: weak type}.
\end{proof}

\subsection{Sparse domination of $Z$}

In this section, we will establish the sparse domination,  Theorem \ref{T: sparse decomposition}. Before we begin, let us make the following observation: 

\begin{remark}\label{nonincreasingW}
If a process $W_t$  solves the homogeneous equation,
$$\mathd W_t=(V_{t_-}-a\Id)W_{t_-}\mathd t,$$
then $|W_t|$ is non-increasing.  Indeed,
\[
\mathd \langle W_t, W_t\rangle =2\langle \mathd W_t,W_t\rangle =2\langle (V_t-a\Id)W_t ,W_t\rangle \mathd t .
\]
But since $V_t-a\Id$ takes values in the class of non-positive matrices, the above shows that $|W_t|^2=\langle W_t, W_t\rangle$ is non-increasing 
\end{remark}

To prove Theorem \ref{T: sparse decomposition} we will also need the following corollary which is an easy consequence of Lemma \ref{L: weak type}. .

\begin{corollary}\label{C:WTE}
Let $\widetilde{X}$ and $Z$ be as in Hypothesis \ref{H: hypotheses}.
Then for all $A \in \mathcal{F}_0$ and $\lambda >0$ we have
\begin{equation}
\label{wte2}
\mathbb{P}\left(\{\omega \in A: Z^*(\omega)\vee \widetilde{X}^*(\omega)  \geqslant  \lambda \E(\widetilde{X}|\F_0)(\omega)\} \right) \leqslant 2 \lambda ^{-1} \mathbb{P}(A).
\end{equation}
\end{corollary}

\begin{proof}[Proof of Corollary \ref{C:WTE}]
The corollary follows easily from the weak type estimate in Lemma \ref{L: weak type} if $\mathbb{P}\{\E(\widetilde{X}|\F_0)=0\}=0$ by dividing all processes by $\E(\widetilde{X}|\F_0)$. If $\mathbb{P}\{\E(\widetilde{X}|\F_0)=0\}>0$, let us write
$$X'=\chi_{A\cap \{ \E(\widetilde{X}|\F_0) >0\}}\frac{\widetilde{X}}{\E(\widetilde{X}|\F_0)},$$
$$ Y'=\chi_{A\cap \{ \E(\widetilde{X}|\F_0) >0\}}\frac{Y}{\E(\widetilde{X}|\F_0)}, $$
$$ Z'=\chi_{A\cap \{ \E(\widetilde{X}|\F_0) >0\}}\frac{Z}{\E(\widetilde{X}|\F_0)}.$$
Notice first that $\E(\widetilde{X}|\F_0)$ is of course measurable in $\F_0$, so the assumptions of the weak type estimate in Lemma \ref{L: weak type} are satisfied. We can now estimate 
\begin{equation}\label{prewt}
\mathbb{P}\left(\{\omega : Z'^*(\omega)\vee \widetilde{X}'^*(\omega)  \geqslant  \lambda \} \right) \leqslant 2 \lambda ^{-1} \| \widetilde{X}'\|_1.
\end{equation}
Further notice that, since by assumption $\widetilde{X}\geqslant 0$ and $X\geqslant 0$, we have $\E(\widetilde{X}|\F_0)(\omega) =0 \Rightarrow \widetilde{X}_0(\omega)=0 \Rightarrow X_0(\omega)=0$, hence $X_t(\omega)=0$ for all $t>0$.  In particular, all future increments of $X(\omega)$ are zero and thus also $Y_t(\omega) = 0$ for all $t > 0$. Thanks to Remark \ref{nonincreasingW} we see that $Z_t(\omega)$ also remains $0$.  This means, the part 
$\{\E(\widetilde{X}|\F_0)=0\}$ takes no effect on the estimate we seek to prove, which allows us to deduce the claim from estimate \eqref{prewt}.
\end{proof}

\begin{proof}[Proof of Theorem \ref{T: sparse decomposition}]
Now that we have an appropriate distributional inequality \eqref{wte2}, we are ready to establish a sparse domination.
We consider for the moment continuous path processes. To prove \eqref{sparse dec}, we will prove that for all $t\geqslant 0$ we have
\begin{equation}
|Z_t| \leqslant 4 S(\widetilde{X}).
\label{est_t}
\end{equation}

To do so, we will successively construct a sequence of filtrations and an increasing sequence of stopping times $(T^j)_{j=0}^\infty$ to determine the sparse domination. 

The key of the proof, besides the weak type estimate established in Lemma \ref{L: weak type} and Corollary \ref{C:WTE}, relies on recursivity in order to construct the stopping times. The construction of the sparse domination differs from the one in the case of $Y$ because the operator we want to estimate is a perturbation of $Y$ (which is differentially subordinate to $X$).
Indeed, the iteration procedure produces certain difficulties concerning the foot of the iterated processes. The difficulty arises because $Z$ is a process with infinite memory and governed by a stochastic ODE taking into account the value of $Z$. We cannot just go ahead and set its foot back to 0 or to anything else, that will allow us to use the weak type estimate in Lemma \ref{L: weak type}. Instead, we build an array of processes that we dominate in a different way, making use of the homogeneous part of the ODE via Remark \ref{nonincreasingW}.

Now, we begin with the domination. First let $\mathfrak{F}^0=\mathfrak{F}$, i.e.\@ $\F^0_t=\F_t$ and  let 
$Y_t^0=Y_t$, $\widetilde{X}_t^0=\widetilde{X}_t$, $Z_t^0=Z_t$. Also we will put $T^0=0$ and hence $E_0=\{T^0<\infty\}=\Omega$.  

Define the set
\[E_1=\{ \omega \in \Omega : Z^{0*}(\omega) \vee \widetilde{X}^{0*}(\omega)>4\E(\widetilde{X}|\F_0)(\omega)\}. \]
Obviously, $Y^0$, $\widetilde{X}^0$ and $Z^0$ satisfy the assumptions of Corollary \ref{C:WTE}, 

so we can apply it with $\lambda=4$ to estimate $\mathbb{P}(E_1)$ and more precisely if $A_0 \in \F_0$ then 
\begin{equation*}
\mathbb{P}(A_0\cap E_1) \leqslant \dfrac{2}{4} \mathbb{P}(A_0)=\dfrac{1}{2}\mathbb{P}(A_0).
\end{equation*}
We can also associate a stopping time
\[T^1(\omega)=\inf\{t>0:|Z_t^{0}(\omega)| \vee \widetilde{X}_t^{0}(\omega)>4\E(\widetilde{X}|\F_0)(\omega)\},\]
which is by definition almost surely finite in $E_1$. 
Now if $t\in [0,T^1)$ then by definition of $T^1$ we have
\[ |Z^0_t(\omega)|\chi_{[T^{0},T^{1})}(t)\chi_{E_0}(\omega) \leqslant 4 \E(\widetilde{X}|\mathcal{F}_0)(\omega)\chi_{E_0}(\omega). 
\]
Indeed, more is true. We define the process $\widetilde{Z}^0_t$ to obey $\widetilde{Z}^0_0=0$ and $$\mathd \widetilde{Z}^0_t=(V_t-a \Id)\widetilde{Z}^0_t \mathd t + \mathd Y_t \chi_{[0,T^1)}.$$
Thanks to Remark \ref{nonincreasingW}, we know that once $T^1$ is reached and $\mathd Y$ has been stopped, the ODE above becomes homogeneous with initial condition $Z_{T^1}$ and its solution does not increase in norm. For the process $\widetilde{Z}^0$ we have therefore an estimate not only for times $t < T^1$ but for all times $t\geqslant 0$, namely:
\[
|\widetilde{Z}^0_t(\omega)| \leqslant 4 \E(\widetilde{X}|\mathcal{F}_0)(\omega)\chi_{E_0}(\omega).
\]

When $T^1=\infty$ we are done, because on $E_0 \setminus E_1$ our sparse estimate is complete. For the next step, we work on $E_1=\{T^1<\infty\}$.
We define the filtration $\mathfrak{F}^1 = (\F^1_t)_{t \geqslant0} \assign (\F^0_{t\vee T^1})_{t\geqslant0}=(\F_{t\vee T^1})_{t\geqslant0}$ and consider the following adapted processes
\begin{align*}
&Y_t^{1}=\chi_{E_1}\Big(0 + \int_{T^1_+}^{t \vee T^1}\mathd Y_s\Big)=\chi_{E_1}(Y_{t \vee T^1}-Y_{T^1}), \\
&\widetilde{X}_t^{1}=\chi_{E_1}\Big(\widetilde{X}_{T^1}+\int_{T^1_+}^{t \vee T^1}\mathd \widetilde{X}_s\Big ), 
\end{align*}
with similar definitions for $X^1$ and $A^1$. Notice that $Y^1$ is differentially subordinate to $X^1$. 
Define $Z^1$ adapted to $\mathfrak{F}^1$ by its initial value $Z^1_0=0$ and satisfying for $t>0$ the stochastic differential equation
\begin{equation*}\mathd Z^{1}_t=(V_{t_-}-a\Id)Z^{1}_{t_-} \mathd t + \mathd Y^{1}_t\end{equation*}
and the set    
\[E_2=\{ \omega \in E_1 : Z^{1*}(\omega)\vee \widetilde{X}^{1*}(\omega)>4\E(\widetilde{X}|\F_{T^1})(\omega)\} \]
with its corresponding stopping time
\[T^2(\omega)=\inf\{t>T^1:|Z_t^1(\omega)|\vee \widetilde{X}_t^1(\omega)>4\E(\widetilde{X}|\F_{T^1})(\omega)\}.\]

Observe that the newly defined processes $X^1$, $Y^1$, $A^1$, $\widetilde{X}^1$ and $Z^1$ again satisfy the assumptions of Lemma \ref{L: weak type} and Corollary \ref{C:WTE} (but now with respect to the filtration $\mathfrak{F}^1$). It is important to notice that in particular $|Z^1_0|\leqslant |\widetilde{X}^1_0|$ so we can apply the weak type estimate \eqref{wte2} to obtain
\[\mathbb{P}(E_2) \leqslant \frac 12 \mathbb{P}(E_1).\]
Moreover, applying \eqref{wte2} to the set $A=A_1$, we get for every $A_1\subseteq E_1$, $A_1 \in \F_{T^1}$
\[\mathbb{P}(A_1 \cap E_2)=\mathbb{P}(\{\omega \in A_1:Z^{1*}(\omega)\vee \widetilde{X}^{1*}(\omega)>4\E(\widetilde{X}^1|\F_{T^1})(\omega) \}) \leqslant \frac{1}{2}\mathbb{P}(A_1).\]
Since $Z_t^1$ has support on $E_1$, by definition of $T^2$ we have for $t<T^2$
\[ |Z_t^{1}| \leqslant 4 \E(\widetilde{X}|\mathcal{F}_{T^1}) \chi_{E_1}.\]
Finally, we define the process $\widetilde{Z}^1_t$ on all of $\Omega$ via its initial value 0 as well as the ODE 
\begin{equation*}\mathd \widetilde{Z}^{1}_t=(V_{t_-}-a\Id)\widetilde{Z}^{1}_{t_-} \mathd t + \mathd Y^{1}_t\chi_{[T^1,T^2)}. \end{equation*}
Using Remark \ref{nonincreasingW} we have $\widetilde{Z}^1_t=0$ outside $E_1$. We also have $\widetilde{Z}^1_t=0$ for $t<T^1$ and non-increasing after $T^2$. Clearly $\widetilde{Z}^1_t$ is identical to $Z^1$ up until $T^2$. So we have 
$$ |\widetilde{Z}^1_t|\leqslant 4 \mathbb{E}(\widetilde{X}|\F_{T^{1}})\chi_{E_1}.$$
and the estimate holds for all $t$. 
If $T^2=\infty$ then we are done, because the sparse estimate is complete for $E_0 \setminus E_2$. Otherwise, for the next step, we work on $E_2=\{T^2<\infty\}$, defining processes with a superscript 2 as before. 

Repeating the procedure inductively delivers processes $\widetilde{Z}^n_t$ all with initial value 0 and satisfying the ODE 
$$ \mathd \widetilde{Z}^{n}_t=(V_{t_-}-a\Id)\widetilde{Z}^{n}_{t_-} \mathd t + \mathd Y^{n}_t\chi_{[T^n,T^{n+1})}.$$ These processes are defined on all of $\Omega$, and satisfy for all $t\geqslant 0$
$$|\widetilde{Z}^n_t| \leqslant 4 \mathbb{E}(\widetilde{X}|\F_{T^{n}}) \chi_{E_n}.$$ 

To finish, we sum the ODEs to obtain $\sum_{n=0}^{\infty} \widetilde{Z}_0^n=0$ and
$$ \mathd \Big( \sum_{n=0}^{\infty} \widetilde{Z}^n\Big)_t=(V_{t_-}-a \Id) \Big( \sum_{n=0}^{\infty} \widetilde{Z}^n\Big)_{t_-} \mathd t + \Big(\sum_{n=0}^{\infty} \chi_{[T^i,T^{i+1})}\Big) \mathd Y_t ,$$
where the last term is simply $\mathd Y_t$. 
The infinite sums are justified because $\mathbb{P}(E_n)\to 0$.
By uniqueness of the solution to the original initial value problem, we have $$Z_t= \sum_{n=0}^{\infty}  \widetilde{Z}_t^n.$$ We can now deduce the domination $$|Z_t| \leqslant 4\sum_{n=0}^{\infty}\mathbb{E}(\widetilde{X}|\F_{T^{n}}) \chi_{E_n}$$ as claimed.

\end{proof}

\subsubsection{Changes in the presence of jumps}
We already have a weak type estimate in the presence of jumps. The additional changes required to the domination itself are similar to what we have seen in the martingale case with jumps. We show that $|Z_t|\leqslant 8 S(\widetilde{X})$. It is important to notice that $Y$ has a jump $Y_T-Y_{T_-}$ at time $T$ if and only if $Z$ does with $Z_T-Z_{T_-}=Y_T-Y_{T_-}$. As before, there exists an operator $r_T$ of norm at most 1 so that $r_T(X_T-X_{T_-})=|X_T-X_{T_-}|$. Therefore, if a jump arises at the stopping time $T^j$, we may lose control of the size of $Z$. This is counter acted in writing $$Z_{T^j}=Z_{T^j_-}+(Z_{T^j}-Z_{T^j_-})=Z_{T^j_-}+(Y_{T^j}-Y_{T^j_-}).$$ 
Now, using differential subordination 

$$|Z_{T^j}|\leqslant |Z_{T^j_-}| +  |X_T-X_{T_-}|  =|Z_{T^j_-}| + r_{T^j}(X_{T^j}-X_{T^j_-}).$$

Before $T^j$, we have domination on $|Z_{T^j_-}|$ and $r_{T^j}X_{T^j_-}$.
The remaining term $r_{T^j}X_{T^j}$ is put with the new iterates $X^j,Y^j,Z^j$ as follows: 
\begin{align*}
& X^j_t=\chi_{E_j}\Big( X_{T^j}+\int_{T^j_+}^{t \vee T^j}\mathd X_s\Big),\\
& Y^j_t=\chi_{E_j}\Big( r_{T^j}X_{T^j}+\int_{T^j_+}^{t \vee T^j}\mathd Y_s\Big),\\
& Z^j_t=\chi_{E_j}\Big( r_{T^j}X_{T^j}+\int_{T^j_+}^{t \vee T^j}\mathd Z_s\Big).
\end{align*}
The remaining part of the argument remains the same.

\subsection{$L^p$ estimates}

In this section we provide $L^p$ estimates for the process $Z$.
\begin{proof}[Proof of Theorem \ref{Z est}]
This follows from the sparse domination in Theorem \ref{T: sparse decomposition} and the corresponding estimates for the sparse operator $S(\widetilde{X})$. For every $p>1$ there exists $C_p'>0$ such that
\begin{equation}
\|S(\widetilde{X})\|_{\textup{L}^p} \leqslant C_p' \|\widetilde{X}\|_{\textup{L}^p}.
\label{SO1}
\end{equation}
Moreover, assuming that the weight $w$ is a positive, uniformly integrable submartingale, the estimate
\begin{equation}
\|S(\widetilde{X})\|_{\textup{L}^2(w)} \leqslant \tilde{C}_2'Q_2^{\F}(w) \|\widetilde{X}\|_{\textup{L}^2(w)}
\label{SO2}
\end{equation}
also holds for some constant $\tilde{C}_2'>0$.
 
We omit the proof of \eqref{SO1} and \eqref{SO2} as they are almost the same as the proof of the estimate for the sparse operator in the martingale case.  Namely for \eqref{SO1} the proof without weight, and for \eqref{SO2} the proof with the weights $w_{T^j}$ sampled at stopping times $T^j$ replaced by the projections $\mathbb{E}(w|\F_{T^j})$. Observe that if $a=0$, then $w$ and $\widetilde{X}$ are actually martingales so we get exactly the same estimate as before.

Now \eqref{sparse dec}, \eqref{SO1} and \eqref{SO2} obviously imply \eqref{Zest_1} and \eqref{Zest_2} which proves the theorem.
\end{proof}

\section{Estimates for the Riesz vector}
The proof of Theorems \ref{unweightedRiesz} and \ref{Riesz} now follows standard arguments.

We start with the proof of \eqref{E1w} in Theorem  \ref{Riesz}. Since here $\Ric_\varphi \geqslant 0$ (i.e.\@ $a=0$), as we mentioned before, both $\widetilde{X}=X$ and the weight $w$ are martingales.  Hence we are able to use the martingale extrapolation theorem from \cite{DPW}. Using this theorem for $r=2$ we extrapolate from \eqref{SO2} and obtain weighted $L^p$ estimates for $S(\widetilde{X})$
\begin{equation*}
\|S(\widetilde{X})\|_{\textup{L}^p(w)} \leqslant \tilde{c}_p'Q_p^{\F}(w)^{\max\{1,\frac{1}{p-1}\}} \|\widetilde{X}\|_{\textup{L}^p(w)},
\end{equation*}
where $Q_p^{\mathcal{F}}(w)$ is the $A_p$ characteristic of the weight $w$ defined on the filtered space as
\[Q_p^{\mathcal{F}}(w)=\sup_\tau \big\|\E(w|\F_{\tau})\E(u|\F_{\tau})^{p-1}\big\|_{\infty}.\]
In the above expression the supremum runs over adapted stopping times $\tau$ and by $u$ we denote the dual weight satisfying $u^pw=u$.

Obviously this estimate along with the sparse domination \eqref{sparse dec} implies a weighted $\textup{L}^p$ estimate for $Z$
\begin{equation}
\|Z^*\|_{\textup{L}^p(w)} \leqslant \tilde{c}_pQ_p^{\F}(w)^{\max\{1,\frac{1}{p-1}\}} \|\widetilde{X}\|_{\textup{L}^p(w)},
\label{Zest_3}
\end{equation}
for some constant $\tilde{c}_p>0$.

Now by taking the probabilistic representation of the Riesz transform \eqref{rep2}, one can write
\begin{align}
\|\vec{R}_\varphi f\|^p_{\textup{L}^p(w)} &=\int_M |\vec{R}_\varphi f|^p w(x)d\mu_\varphi (x) \nonumber \\
&= 2^p\lim_{y\to\infty}\int_M \big|\mathbb{E}_y\left[Z_\tau|B^M_{\tau}=x \right]\big|^p w(x)d\mu_\varphi (x) \nonumber\\
&\leqslant 2^p\lim_{y\to\infty}\int_M \mathbb{E}_y\big(|Z_\tau|^p|B^M_{\tau}=x\big) w(x)d\mu_\varphi (x) \nonumber\\
&= 2^p\lim_{y\to\infty}\int_M \mathbb{E}_y\left[ |Z_\tau|^p w_\tau   |B^M_{\tau}=x \right]d\mu_\varphi(x). \label{riesz1}
\end{align}
Here in the third line we applied the Jensen's inequality to the conditional expectation, while in the fourth line we used that for every $\omega$ such that $B^M_\tau(\omega)=x$ we have $w_\tau(\omega)=w(x)$.

Since $d\mu_\varphi(x)$ is the invariant measure with respect to $B^M_t$, i.e.\@ $d\mu_\varphi(x)$ is counting how many trajectories arrive at $x$, we have
\[d\mu_\varphi(x)=d\mathbb{P}(B^M_\tau=x)\]
so we can write \eqref{riesz1} as
\begin{align*}
2^p\lim_{y\to\infty}\mathbb{E}\left( \mathbb{E}_y\left[ |Z_\tau|^p w_\tau   |B^M_{\tau}=x \right]\right)&= 2^p\lim_{y\to\infty}\mathbb{E}\big(|Z_\tau|^p w_\tau\big)\\
&= 2^p\lim_{y\to\infty}\int_\Omega |Z_\tau(\omega)|^p w_\tau(\omega) d\mathbb{P}(\omega)\\
&= 2^p\lim_{y\to\infty}\|Z_\tau\|^p_{\textup{L}^p(w)}.
\end{align*}
Now we use the estimate $\eqref{Zest_3}$ to get
\begin{align*}
2^p\lim_{y\to\infty}\|Z_\tau\|^p_{\textup{L}^p(w)} &\leqslant 2^p\tilde{c}_p^p\lim_{y\to\infty}Q_p^{\F(y)}(w)^{\max\{p,\frac{p}{p-1}\}}\|\widetilde{X}\|^p_{\textup{L}^p(w)} \\
&\leqslant 2^p\tilde{c}_p^p\lim_{y\to\infty}Q_p^{\F(y)}(w)^{\max\{p,\frac{p}{p-1}\}}\|Qf(B^M_{\tau},B_{\tau})\|_{\textup{L}^p(w)}^p  \\
& \leqslant \tilde{C}_p  \tilde{Q}_p(w)^{\max\{p,\frac{p}{p-1}\}} \|f\|_{\textup{L}^p(w)}^p,
\end{align*}
for some constant $\tilde{C}_p>0$ depending only on $p$. Combining all of the above we finally get
\[\|\vec{R}_\varphi f\|^p_{\textup{L}^p(w)}\leqslant  \tilde{C}_p  \tilde{Q}_p(w)^{\max\{p,\frac{p}{p-1}\}} \|f\|_{\textup{L}^p(w)}^p,\]
which by taking the $p$-th root gives us exactly \eqref{E1w}.

Notice that the sparse domination itself depends upon the used filtration (and hence $y$). Here the norm $\|\widetilde{X}\|_{\textup{L}^p(w)}$ is taken at $t=\infty$, which is $\tau$ in our stopped processes. We also use the fact that for $\omega$ such that $(B^M_{\tau}(\omega),B_{\tau}(\omega))=(x,0)$ we have
\[Qf(B^M_{\tau}(\omega),B_{\tau}(\omega))=Qf(x,0)=f(x)\]
and hence $\|Qf(B^M_{\tau},B_{\tau})\|_{\textup{L}^p(w)}=\|f\|_{\textup{L}^p(w)}$. Also $Q^{\mathcal{F}^{(y)}}_p(w)$ is the $A_p$ characteristic that corresponds to the filtration when $B_0=y$ and $\tilde{Q}_p(w)$ is the Poisson flow characteristic. When the weight $w$ is a martingale, those two characteristics are comparable (for more details see \cite{DPW}).

In the negative curvature case, i.e. $a>0$, we don't have the needed comparability anymore and that is why we are not able to prove the weighted estimate \eqref{E1w} using this method. By using the same method we could get some sort of weighted estimate, but the Poisson flow characteristic $\tilde{Q}_p(w)$ should be replaced with some different characteristic which would be comparable to the $A_p$ characteristic $Q^{\mathcal{F}^{(y)}}_p(w)$.

On the other hand, unweighted estimates from Theorem \ref{unweightedRiesz} hold even in the case when $a>0$ (negative curvature). Estimate \eqref{E1} is proved the same way as \eqref{E1w} just by omitting the weight $w$.

\section{Examples}
The setting we are working on is very general and covers a lot of classical examples. In this section we are only going to mention a few of them.
\begin{itemize}
\item[(1)] (Euclidean space) Taking $M=\R^n$ and $\varphi=0$, we get $\Ric_\varphi=0$ and $M_t=\textup{Id}$ for all $t\geqslant 0$. This way one can retrieve estimates for the classical Riesz transforms  $R_i = \frac{\partial}{\partial x_i} (-\Delta)^{-1/2}$ on $\R^n$.
\item[(2)] (Gauss space) Taking $M=\R^n$ and
$$\varphi(x)=\frac{\|x\|^2}{2}+\frac{n}{2}\log(2\pi),$$
we get $\mu_\varphi=\gamma_n$ (Gaussian measure) and $M_t=e^{-t}\textup{Id}$ for all $t\geqslant 0$. This gives us the estimates for the Riesz transform associated with the Ornstein-Uhlenbeck operator $\Delta_\varphi=\Delta-x\cdot\nabla$ on $(\R^n,\gamma_n)$.
\item[(3)] (Bessel setting) By taking $M=(0,\infty)$ and $\varphi(x)=\ln(x^{-2\alpha})$ for some $\alpha\geqslant 0$ we get
\[\Delta_\varphi=\frac{d^2}{dx^2}+\frac{2\alpha}{x}\frac{d}{dx}=-B_\alpha,\]
where $B_\alpha$ is the Bessel operator considered by Muckenhoupt and Stein (\cite{MS}). The measure is now $d\mu_\varphi(x)=x^{2\alpha}dx$ and $\Ric_\varphi=\frac{2\alpha}{x^2}\geqslant 0$. Hence we also retrieve estimates for the Riesz transform $R_\alpha$ associated with $B_\alpha$ defined as $R_\alpha=\frac{d}{dx}B_\alpha^{-1/2}$.

Bessel-Riesz transforms in higher dimensions were considered in \cite{BetCC}. Given a multi-index $\alpha=(\alpha_1,\alpha_2,\dots,\alpha_n)$ $(\alpha_i \geqslant 0)$ one defines the $n$-dimensional Bessel differential operator as 
\[B_\alpha=\sum_{i=1}^n \left(-\frac{\partial^2}{\partial x_i^2}-\frac{2\alpha_i}{x_i}\frac{\partial}{\partial x_i}\right)\]
and the Riesz transforms associated with $B_\alpha$ as
\[R_{\alpha,i} = \frac{\partial}{\partial x_i} (B_\alpha)^{-1/2}.\]
Now by taking $M=(0,\infty)^n$ and $\varphi(x)=\sum_{i=1}^n\ln(x^{-2\alpha_i})$ we see that again we get 
\[\Delta_\varphi=-B_\alpha \quad \text{and} \quad d\mu_\varphi(x)=x^{2\alpha}dx,\]
where $x^{2\alpha}$ denotes $x_1^{2\alpha_1}\cdots x_n^{2\alpha_n}$. Therefore we also get dimensionless estimates for the vector Riesz transform $R_\alpha$ on $(\R_+^n,x^{2\alpha}dx)$.
\end{itemize}

\end{document}